\documentclass[fleqn]{article}
\usepackage[a4paper, left=2cm, right=2cm, top=3cm, bottom=3cm]{geometry}
\usepackage[ansinew]{inputenc}
\usepackage[T1]{fontenc}
\usepackage{tocbibind}
\usepackage{hyperref}
\usepackage{ifthen}
\usepackage{amsmath}
\usepackage{amssymb}
\usepackage{amsthm}
\usepackage{upgreek}
\usepackage[all]{xy}
\theoremstyle{definition}
\newtheorem*{proposition*}{Proposition}
\newtheorem*{theorem*}{Theorem}

\allowdisplaybreaks 

\newcommand{\BIGOP}[1]{\mathop{\mathchoice%
{\raise-0.22em\hbox{\huge $#1$}}%
{\raise-0.05em\hbox{\Large $#1$}}%
{\hbox{\large $#1$}}%
{#1}}}

\def\morchoice#1#2{\begingroup\setbox0=\hbox{$#1\xrightarrow{#2}$}%
        \setbox1=\hbox{$#1\longrightarrow$}%
        \ifdim\wd0<\wd1
        \stackrel{#2}\longrightarrow
        \else
        \xrightarrow{#2}\fi\endgroup}

\newcommand{\booktitle}[1]{\textsl{#1}}											
\newcommand{\eigenname}[1]{\textsc{#1}}											
\newcommand{\newnotion}[1]{\textit{#1}}											

\DeclareMathOperator{\DiagonalFunctor}{\mathrm{Diag}}				
\DeclareMathOperator{\Mor}{\mathrm{Mor}}										
\DeclareMathOperator{\Ob}{\mathrm{Ob}}											
\DeclareMathOperator{\TotalSimplicialObject}{\mathrm{Tot}}	
\newcommand{\ascendinginterval}[2]{\lceil #1, #2 \rceil}		
\newcommand{\AssociatedComplex}[1][]{\ifthenelse{\equal{#1}{}}{\mathrm{C}}{\mathrm{C}^{(#1)}}}	
\newcommand{\bigtimes}{\BIGOP{\times}}											
\newcommand{\Cat}{\mathbf{Cat}}															
\newcommand{\CatNaturalNumber}[1]{[#1]}											
\newcommand{\ClassifyingSimplicialSet}[1][]{\ifthenelse{\equal{#1}{}}{\mathrm{B}}{\mathrm{B}^{(#1)}}}	
\newcommand{\codegeneracy}{\upsigma}												
\newcommand{\coface}{\updelta}															
\newcommand{\comp}{\circ}																		
\newcommand{\cotime}{\uptau}																
\newcommand{\csCategory}[1][]{\mathbf{cs}^{#1}}							
\newcommand{\degeneracy}{\mathrm{s}}												
\newcommand{\descendinginterval}[2]{\lfloor #1, #2 \rfloor}	
\newcommand{\EilenbergWFunctor}{\mathrm{W}}									
\newcommand{\face}{\mathrm{d}}															
\newcommand{\Grps}{\mathbf{Grps}}														
\newcommand{\HomologyGroup}{\mathrm{H}}											
\newcommand{\homotopic}{\sim}																
\newcommand{\homotopyequivalent}{\simeq}										
\newcommand{\id}{\mathrm{id}}																
\newcommand{\ins}{\mathrm{ins}}															
\newcommand{\Integers}{\mathbb{Z}}													
\newcommand{\isomorphic}{\cong}															
\newcommand{\KanClassifyingSimplicialSet}{\overline{\EilenbergWFunctor}}	
\newcommand{\lcb}{\pmb{(}}																	
\newcommand{\map}{\rightarrow}															
\newcommand{\morphism}[1][]{\mathpalette\morchoice{#1}}			
\newcommand{\Naturals}{\mathbb{N}}													
\newcommand{\Nerve}{\mathrm{N}}															
\newcommand{\op}{\mathrm{op}}																
\newcommand{\rcb}{\pmb{)}}																	
\newcommand{\sCategory}[1][]{\mathbf{s}^{#1}}								
\newcommand{\Sets}{\mathbf{Sets}}														
\newcommand{\sGrps}{\mathbf{sGrps}}													
\newcommand{\SimplexTypes}{\mathbf{\Delta}}									
\newcommand{\Spl}{\mathrm{Spl}}															
\newcommand{\sSets}[1][]{\sCategory[#1]\Sets}								
\newcommand{\StandardSimplex}[1]{\Delta^{#1}}								

\title{The functors \(\KanClassifyingSimplicialSet\) and \(\DiagonalFunctor \comp \mathrm{Nerve}\) \\ are simplicially homotopy equivalent}
\author{Sebastian Thomas}
\date{March 17th, 2008}

\begin{document}

\maketitle

\renewcommand{\thefootnote}{\fnsymbol{footnote}}
\footnotetext[0]{Mathematics Subject Classification 2000: 18G30, 55U10.}
\renewcommand{\thefootnote}{\arabic{footnote}}

\begin{abstract}
Given a simplicial group \(G\), there are two known classifying simplicial set constructions, the Kan classifying simplicial set \(\KanClassifyingSimplicialSet G\) and \(\DiagonalFunctor \Nerve G\), where \(\Nerve\) denotes the dimensionwise nerve. They are known to be weakly homotopy equivalent. We will show that \(\KanClassifyingSimplicialSet G\) is a strong simplicial deformation retract of \(\DiagonalFunctor \Nerve G\). In particular, \(\KanClassifyingSimplicialSet G\) and \(\DiagonalFunctor \Nerve G\) are simplicially homotopy equivalent.
\end{abstract}

\section{Introduction}

We suppose given a simplicial group \(G\). \eigenname{Kan} introduced in \cite{kanhtcg} the Kan classifying simplicial set \(\KanClassifyingSimplicialSet G\). The functor \(\KanClassifyingSimplicialSet\) from simplicial groups to simplicial sets is the right adjoint, and actually the homotopy inverse, to the Kan loop group functor, which is a combinatorial analogue to the topological loop space functor. Alternatively, dimensionwise application of the nerve functor for groups yields a bisimplicial set \(\Nerve G\), to which we can apply the diagonal functor to obtain a simplicial set \(\DiagonalFunctor \Nerve G\). The latter construction is used for example by \eigenname{Quillen} \cite[appendix Q.3]{fmfothoav} and \eigenname{Jardine} \cite[p.\ 41]{jarahtgakt}.

It is well-known that these two variants \(\KanClassifyingSimplicialSet G\) and \(\DiagonalFunctor \Nerve G\) for the classifying simplicial set of \(G\) are weakly homotopy equivalent. Better still, the Kan classifying functor \(\KanClassifyingSimplicialSet\) can be obtained as the composite of the nerve functor with the total simplicial set functor \(\TotalSimplicialObject\) as introduced by \eigenname{Artin} and \eigenname{Mazur} \cite{amotvkt} (\footnote{This is not the total simplicial set as used by \eigenname{Bousfield} and \eigenname{Friedlander} \cite[appendix B, p.\ 118]{bfhtogssabs}.}); and \eigenname{Cegarra} and \eigenname{Remedios} \cite{crtrbtdatbcoabs} showed that already the total simplicial set functor and the diagonal functor, applied to a bisimplicial set, yield weakly homotopy equivalent results. Moreover, the model structures on the category of bisimplicial sets induced by \(\TotalSimplicialObject\) resp.\ by \(\DiagonalFunctor\) are related \cite{crtbotwcothtobs}.

The aim of this article is to prove the following

\begin{theorem*}
The Kan classifying simplicial set \(\KanClassifyingSimplicialSet G\) is a strong simplicial deformation retract of \(\DiagonalFunctor \Nerve G\). In particular, \(\KanClassifyingSimplicialSet G\) and \(\DiagonalFunctor \Nerve G\) are simplicially homotopy equivalent.
\end{theorem*}

This commutativity up to simplicial homotopy equivalence fits into the following diagram.

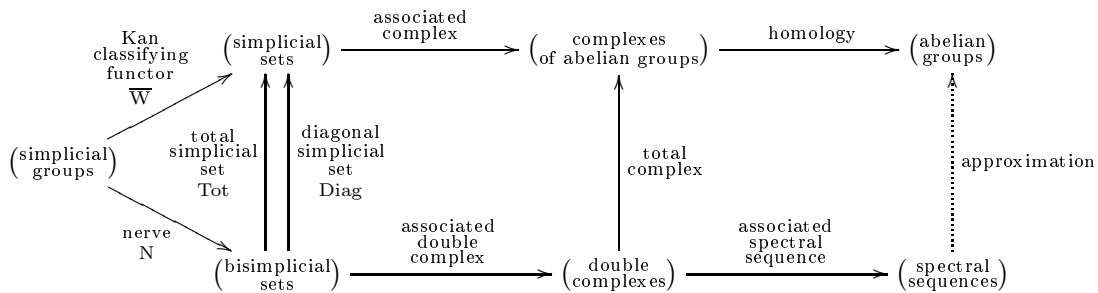
\[\xymatrix@C=10mm{
& \big( \substack{\text{simplicial} \\ \text{sets}} \big) \ar[rr]^(0.42){\substack{\text{associated} \\ \text{complex}}} & & \big( \substack{\text{complexes} \\ \text{of abelian groups}} \big) \ar[rr]^(0.58){\text{homology}} & & \big( \substack{\text{abelian} \\ \text{groups}} \big) \\
\big( \substack{\text{simplicial} \\ \text{groups}} \big) \ar[ru]^*+<0ex,-1ex>{\substack{\text{Kan} \\ \text{classifying} \\ \text{functor} \\[0.5ex] \KanClassifyingSimplicialSet}} \ar[rd]_*+<2ex,0.5ex>{\substack{\text{nerve} \\[0.5ex] \Nerve}} & & & & & \\
& \big( \substack{\text{bisimplicial} \\ \text{sets}} \big) \ar@<-1ex>[uu]_{\substack{\text{diagonal} \\ \text{simplicial} \\ \text{set} \\[0.5ex] \DiagonalFunctor}} \ar@<1ex>[uu]^{\substack{\text{total} \\ \text{simplicial} \\ \text{set} \\[0.5ex] \TotalSimplicialObject}} \ar[rr]^{\substack{\text{associated} \\ \text{double} \\ \text{complex}}} & & \big( \substack{\text{double} \\ \text{complexes}} \big) \ar[uu]_{\substack{\text{total} \\ \text{complex}}} \ar[rr]^{\substack{\text{associated} \\ \text{spectral} \\ \text{sequence}}} & & \big( \substack{\text{spectral} \\ \text{sequences}} \big) \ar@.[uu]_{\text{approximation}}
}\]

By definition, the homology of a simplicial group is obtained by composition of the functors in the upper row. The generalised Eilenberg-Zilber theorem (due to \eigenname{Dold}, \eigenname{Puppe} and \eigenname{Cartier} \cite[Satz 2.9]{dphnaf}) states that the quadrangle in the middle of the diagram commutes up to homotopy equivalence of complexes. The composition of the functors in the lower row yields the Jardine spectral sequence \cite[Lemma 4.1.3]{jarahtgakt} of \(G\), which has \(E_{p, n - p}^1 \isomorphic \HomologyGroup_{n - p}(G_p)\), and which converges to the homology of \(G\). Similarly for cohomology.

\subsection*{Conventions and notations}

We use the following conventions and notations.

\begin{itemize}
\item The composite of morphisms \(X \morphism[f] Y\) and \(Y \morphism[g] Z\) is denoted by \(X \morphism[f g] Z\). The composite of functors \(\mathcal{C} \morphism[F] \mathcal{D}\) and \(\mathcal{D} \morphism[G] \mathcal{E}\) is denoted by \(\mathcal{C} \morphism[G \comp F] \mathcal{E}\).
\item If \(\mathcal{C}\) is a category and \(X, Y \in \Ob \mathcal{C}\) are objects in \(\mathcal{C}\), we write \({_\mathcal{C}}(X, Y) = \Mor_\mathcal{C}(X, Y)\) for the set of morphisms between \(X\) and \(Y\). Moreover, we denote by \(\lcb \mathcal{C}, \mathcal{D} \rcb\) the functor category that has functors between \(\mathcal{C}\) and \(\mathcal{D}\) as objects and natural transformations between these functors as morphisms.
\item Given a functor \(I \morphism[X] \mathcal{C}\), we sometimes denote the image of a morphism \(i \morphism[\theta] j\) in \(I\) by \(X_i \morphism[X_\theta] X_j\). This applies in particular if \(I = \SimplexTypes^\op\) or \(I = \SimplexTypes^\op \times \SimplexTypes^\op\).
\item We use the notations \(\Naturals = \{1, 2, 3, \dots\}\) and \(\Naturals_0 = \Naturals \cup \{0\}\).
\item Given integers \(a, b \in \Integers\), we write \([a, b] := \{z \in \Integers \mid a \leq z \leq b\}\) for the set of integers lying between \(a\) and \(b\). Moreover, we write \(\ascendinginterval{a}{b} := (z \in \Integers \mid a \leq z \leq b)\) for the \newnotion{ascending interval} and \(\descendinginterval{a}{b} = (z \in \Integers \mid a \geq z \geq b)\) for the \newnotion{descending interval}. Whereas we formally deal with tuples, we use the element notation, for example we write \(\prod_{i \in \ascendinginterval{1}{3}} g_i = g_1 g_2 g_3\) and \(\prod_{i \in \descendinginterval{3}{1}} g_i = g_3 g_2 g_1\) or \((g_i)_{i \in \descendinginterval{3}{1}} = (g_3, g_2, g_1)\) for group elements \(g_1, g_2, g_3\).
\item Given an index set \(I\), families of groups \((G_i)_{i \in I}\) and \((H_i)_{i \in I}\) and a family of group homomorphisms \((\varphi_i)_{i \in I}\), where \(\varphi_i\colon G_i \map H_i\) for all \(i \in I\), we denote the direct product of the groups by \(\bigtimes_{i \in I} G_i\) and the direct product of the group homomorphisms by \(\bigtimes_{i \in I} \varphi_i\colon \bigtimes_{i \in I} G_i \map \bigtimes_{i \in I} H_i, (g_i)_{i \in I} \mapsto (g_i \varphi_i)_{i \in I}\).
\end{itemize}

\section{Simplicial preliminaries}

We recall some standard definitions, cf.\ for example \cite{cursht}, \cite{gjsht} or \cite{maysoiat}.

\subsection*{Simplicial objects}

For \(n \in \Naturals_0\), we let \(\CatNaturalNumber{n}\) denote the category induced by the totally ordered set \([0, n]\) with the natural order, and we let \(\SimplexTypes\) be the full subcategory in \(\Cat\) defined by \(\Ob \SimplexTypes := \{\CatNaturalNumber{n} \mid n \in \Naturals_0\}\).

The \newnotion{category of simplicial objects} \(\sCategory \mathcal{C}\) in a given category \(\mathcal{C}\) is defined to be the functor category \(\lcb \SimplexTypes^\op, \mathcal{C} \rcb\). Moreover, the \newnotion{category of bisimplicial objects} \(\sCategory[2] \mathcal{C}\) in \(\mathcal{C}\) is defined to be \(\lcb \SimplexTypes^\op \times \SimplexTypes^\op, \mathcal{C} \rcb\). The dual notion is that of the \newnotion{category} \(\csCategory \mathcal{C} := \lcb \SimplexTypes, \mathcal{C} \rcb\) \newnotion{of cosimplicial objects} in \(\mathcal{C}\).

For \(n \in \Naturals\), \(k \in [0, n]\), we let \(\CatNaturalNumber{n - 1} \morphism[\coface^k] \CatNaturalNumber{n}\) be the injection that omits \(k \in [0, n]\), and for \(n \in \Naturals_0\), \(k \in [0, n]\), we let \(\CatNaturalNumber{n + 1} \morphism[\codegeneracy^k] \CatNaturalNumber{n}\) be the surjection that repeats \(k \in [0, n]\). The images of the morphisms \(\coface^k\) resp.\ \(\codegeneracy^k\) under a simplicial object \(X\) in a given category \(\mathcal{C}\) are denoted by \(\face_k := X_{\coface^k}\), called the \(k\)-th \newnotion{face}, for \(k \in [0, n]\), \(n \in \Naturals\), resp.\ \(\degeneracy_k := X_{\codegeneracy^k}\), called the \(k\)-th \newnotion{degeneracy}, for \(k \in [0, n]\), \(n \in \Naturals_0\). Similarly, in a bisimplicial object \(X\) one defines \newnotion{horizontal} and \newnotion{vertical faces} resp.\ \newnotion{degeneracies}, \(\face_k^\mathrm{h} := X_{\coface^k, \id}\), \(\face_k^\mathrm{v} := X_{\id, \coface^k}\), \(\degeneracy_k^\mathrm{h} := X_{\codegeneracy^k, \id}\), \(\degeneracy_k^\mathrm{v} := X_{\id, \codegeneracy^k}\).

Moreover, we use the ascending and descending interval notation as introduced above for composites of faces resp.\ degeneracies, that is, we write \(\face_{\descendinginterval{j}{i}} := \face_j \face_{j - 1} \dots \face_i\) resp.\ \(\degeneracy_{\ascendinginterval{i}{j}} := \degeneracy_i \degeneracy_{i + 1} \dots \degeneracy_j\).

\subsection*{The nerve}

We suppose given a group \(G\). The \newnotion{nerve} of \(G\) is the simplicial set \(\Nerve G\) given by \(\Nerve_n G = G^{\times n}\) for all \(n \in \Naturals_0\) and by
\[(g_j)_{j \in \descendinginterval{n - 1}{0}} (\Nerve_\theta G) = (\prod_{j \in \descendinginterval{(i + 1) \theta - 1}{i \theta}} g_j)_{i \in \descendinginterval{m - 1}{0}}\]
for \((g_j)_{j \in \descendinginterval{n - 1}{0}} \in \Nerve_n G\) and \(\theta \in {_\SimplexTypes}(\CatNaturalNumber{m}, \CatNaturalNumber{n})\), where \(m, n \in \Naturals_0\).

Since the nerve construction is a functor \(\Grps \morphism[\Nerve] \sSets\), it can be applied dimensionwise to a simplicial group. This yields a functor \(\sGrps \morphism[\Nerve] \sSets[2]\).

\subsection*{From bisimplicial sets to simplicial sets}

We suppose given a bisimplicial set \(X\). There are two known ways to construct a simplicial set from \(X\), namely the diagonal simplicial set \(\DiagonalFunctor X\) and the total simplicial set \(\TotalSimplicialObject X\), see \cite[§ 3]{amotvkt}. We recall their definitions.

The \newnotion{diagonal simplicial set} \(\DiagonalFunctor X\) has entries \(\DiagonalFunctor_n X := X_{n, n}\) for \(n \in \Naturals_0\), while \(\DiagonalFunctor_\theta X := X_{\theta, \theta}\) for \(\theta \in {_\SimplexTypes}(\CatNaturalNumber{m}, \CatNaturalNumber{n})\), where \(m, n \in \Naturals_0\).

To introduce the total simplicial set of \(X\), we define the \newnotion{splitting at \(p \in [0, m]\)} of a morphism \(\CatNaturalNumber{m} \morphism[\theta] \CatNaturalNumber{n}\) in \(\SimplexTypes\) by \(\Spl_p(\theta) := (\Spl_{\leq p}(\theta), \Spl_{\geq p}(\theta))\), where
\[\CatNaturalNumber{p} \morphism[\Spl_{\leq p}] \CatNaturalNumber{p \theta} \text{ and } \CatNaturalNumber{m - p} \morphism[\Spl_{\geq p}] \CatNaturalNumber{n - p \theta}\]
are given by \(i \, \Spl_{\leq p}(\theta) := i \theta\) for \(i \in [0, p]\) and \(i \, \Spl_{\geq p}(\theta) := (i + p) \theta - p \theta\) for \(i \in [0, m - p]\). The \newnotion{total simplicial set} \(\TotalSimplicialObject X\) is defined by
\[\TotalSimplicialObject_n X := \big\{(x_q)_{q \in \descendinginterval{n}{0}} \in \bigtimes_{q \in \descendinginterval{n}{0}} X_{q, n - q} \, \big| \, x_q \face_q^\mathrm{h} = x_{q - 1} \face_0^\mathrm{v} \text{ for all } q \in \descendinginterval{n}{1} \big\} \text{ for } n \in \Naturals_0\]
and by
\[(x_q)_{q \in \descendinginterval{n}{0}} (\TotalSimplicialObject_\theta X) = (x_{p \theta} X_{\Spl_p(\theta)})_{p \in \descendinginterval{m}{0}}\]
for \((x_q)_{q \in \descendinginterval{n}{0}} \in \TotalSimplicialObject_n X\) and \(\theta \in {_\SimplexTypes}(\CatNaturalNumber{m}, \CatNaturalNumber{n})\), where \(m, n \in \Naturals_0\).

There is a natural transformation 
\[\DiagonalFunctor \morphism[\phi] \TotalSimplicialObject,\]
where \(\phi_X\) is given by \(x_n (\phi_X)_n = (x_n \face_{\descendinginterval{n}{q + 1}}^\mathrm{h} \face_{\descendinginterval{q - 1}{0}}^\mathrm{v})_{q \in \descendinginterval{n}{0}}\) for \(x_n \in \DiagonalFunctor_n X\), \(n \in \Naturals_0\), \(X \in \Ob \sSets[2]\); cf.\ \mbox{\cite[formula (1)]{crtrbtdatbcoabs}}.

\subsection*{The Kan classifying simplicial set}

We let \(G\) be a simplicial group. For a morphism \(\theta \in {_{\SimplexTypes}}(\CatNaturalNumber{m}, \CatNaturalNumber{n})\) and non-negative integers \(i \in [0, m]\), \(j \in [i \theta, n]\), we let \(\theta|_{\CatNaturalNumber{i}}^{\CatNaturalNumber{j}} \in {_{\SimplexTypes}}(\CatNaturalNumber{i}, \CatNaturalNumber{j})\) be defined by \(k \theta|_{\CatNaturalNumber{i}}^{\CatNaturalNumber{j}} := k \theta\) for \(k \in \CatNaturalNumber{i}\). \eigenname{Kan} constructed a reduced simplicial set \(\KanClassifyingSimplicialSet G\) by
\[\KanClassifyingSimplicialSet_n G := \bigtimes_{j \in \descendinginterval{n - 1}{0}} G_j \text{ for every } n \in \Naturals_0\]
and
\[(g_j)_{j \in \descendinginterval{n - 1}{0}} \KanClassifyingSimplicialSet_\theta G := (\prod_{j \in \descendinginterval{(i + 1) \theta - 1}{i \theta}} g_j G_{\theta|_{\CatNaturalNumber{i}}^{\CatNaturalNumber{j}}})_{i \in \descendinginterval{m - 1}{0}}\]
for \((g_j)_{j \in \descendinginterval{n - 1}{0}} \in \KanClassifyingSimplicialSet_n G\) and \(\theta \in {_{\SimplexTypes}}(\CatNaturalNumber{m}, \CatNaturalNumber{n})\), see \cite[Definition 10.3]{kanhtcg}. The simplicial set \(\KanClassifyingSimplicialSet G\) will be called the \newnotion{Kan classifying simplicial set} of \(G\).

\subsection*{Notions from simplicial homotopy theory}

For \(n \in \Naturals\), the \newnotion{standard \(n\)-simplex} \(\StandardSimplex{n}\) in the category \(\sSets\) is defined to be the functor \(\SimplexTypes^\op \morphism \Sets\) represented by \(\CatNaturalNumber{n}\), that is, \(\StandardSimplex{n} := {_{\SimplexTypes}}(\bullet, \CatNaturalNumber{n})\). These simplicial sets yield a cosimplicial object \(\StandardSimplex{-} \in \csCategory(\sSets)\). We set \(\face^l := \StandardSimplex{\coface^l} \in {_{\sSets}}(\StandardSimplex{0}, \StandardSimplex{1})\) for \(l \in [0, 1]\).

For a simplicial set \(X\) we define \(\ins_0\) resp.\ \(\ins_1\) to be the composite morphisms
\[X \morphism[\isomorphic] X \times \StandardSimplex{0} \morphism[\id \times \face^1] X \times \StandardSimplex{1} \text{ resp.\ } X \morphism[\isomorphic] X \times \StandardSimplex{0} \morphism[\id \times \face^0] X \times \StandardSimplex{1},\]
where the cartesian product is defined dimensionwise and the isomorphisms are canonical.

For \(k \in [0, n + 1]\), \(n \in \Naturals_0\), we let \(\cotime^k \in \StandardSimplex{1}_n = {_{\SimplexTypes}}(\CatNaturalNumber{n}, \CatNaturalNumber{1})\) be the morphism given by \([0, n - k] \cotime^k = \{0\}\) and \([n - k + 1, n] \cotime^k = \{1\}\). Note that \((x_n) (\ins_0)_n = (x_n, \cotime^0)\) and \((x_n) (\ins_1)_n = (x_n, \cotime^{n + 1})\) for \(x_n \in X_n\).

In the following, we assume given simplicial sets \(X\) and \(Y\).

Simplicial maps \(f, g \in {_{\sSets}}(X, Y)\) are said to be \newnotion{simplicially homotopic}, written \(f \homotopic g\), if there exists a simplicial map \(X \times \StandardSimplex{1} \morphism[H] Y\) such that \(\ins_0 H = f\) and \(\ins_1 H = g\). In this case, \(H\) is called a \newnotion{simplicial homotopy} from \(f\) to \(g\).

The simplicial sets \(X\) and \(Y\) are said to be simplicially homotopy equivalent if there are simplicial maps \(X \morphism[f] Y\) and \(Y \morphism[g] X\) such that \(f g \homotopic \id_X\) and \(g f \homotopic \id_Y\). In this case we write \(X \homotopyequivalent Y\) and we call \(f\) and \(g\) mutually inverse \newnotion{simplicial homotopy equivalences}.

Finally, we suppose given a dimensionwise injective simplicial map \(Y \morphism[i] X\), that is, \(i_n\) is assumed to be injective for all \(n \in \Naturals_0\). We call \(Y\) a \newnotion{simplicial deformation retract} of \(X\) if there exists a simplicial map \(X \morphism[r] Y\) such that \(i r = \id_Y\) and \(r i \homotopic \id_X\). In this case, \(r\) is said to be a \newnotion{simplicial deformation retraction}. If there exists a homotopy \(r i \morphism[H] \id_X\) which is \newnotion{constant along \(i\)}, that is, if \((y_n i_n, \cotime^k) H_n = y_n i_n f_n = y_n i_n g_n\) for \(y_n \in Y_n\), \(k \in [0, n + 1]\), \(n \in \Naturals_0\), then we call \(Y\) a \newnotion{strong} simplicial deformation retract of \(X\) and \(r\) a \newnotion{strong} simplicial deformation retraction.

\section{Comparing \texorpdfstring{$\KanClassifyingSimplicialSet$}{Wbar} and \texorpdfstring{$\DiagonalFunctor \comp \Nerve$}{Diag N}}

We have \(\KanClassifyingSimplicialSet \isomorphic \TotalSimplicialObject \comp \Nerve\). The natural transformation \(\DiagonalFunctor \morphism[\phi] \TotalSimplicialObject\) composed with the nerve functor \(\Nerve\) yields a natural transformation
\[\DiagonalFunctor \comp \Nerve \morphism[D] \KanClassifyingSimplicialSet,\]
given by \((D_G)_n = \bigtimes_{i \in \descendinginterval{n - 1}{0}} \face_{\descendinginterval{n}{i + 1}}\colon \DiagonalFunctor_n \Nerve G \map \KanClassifyingSimplicialSet_n G\) for \(n \in \Naturals_0\) and \(G \in \Ob \sGrps\).

\begin{proposition*}
The natural transformation \(D\) is a retraction. A corresponding coretraction is given by
\[\KanClassifyingSimplicialSet \morphism[S] \DiagonalFunctor \circ \Nerve,\]
where
\[(S_G)_n\colon \KanClassifyingSimplicialSet_n G \map \DiagonalFunctor_n \Nerve G, (g_i)_{i \in \descendinginterval{n - 1}{0}} \mapsto (y_i)_{i \in \descendinginterval{n - 1}{0}}\]
with, defined by descending recursion,
\[y_i := \prod_{j \in \ascendinginterval{i + 1}{n - 1}} (y_j^{- 1} \face_{\descendinginterval{j}{i + 1}} \degeneracy_{\ascendinginterval{i}{j - 1}}) \prod_{j \in \descendinginterval{n - 1}{i}} (g_j \face_{\descendinginterval{j}{i + 1}} \degeneracy_{\ascendinginterval{i}{n - 1}}) \in G_n\]
for each \(i \in \descendinginterval{n - 1}{0}\), \(n \in \Naturals_0\), \(G \in \Ob \sGrps\).
\end{proposition*}
\begin{proof}
We suppose given a simplicial group \(G\). Then we have to show that the maps \((S_G)_n\) for \(n \in \Naturals_0\) commute with the faces and degeneracies of \(G\).

First, we consider the faces. We let \(n \in \Naturals\) and \(k \in [0, n]\). For an \(n\)-tuple \((g_i)_{i \in \descendinginterval{n - 1}{0}} \in \KanClassifyingSimplicialSet_n G\) we compute
\[(g_i)_{i \in \descendinginterval{n - 1}{0}} \face_k (S_G)_{n - 1} = (f_i)_{i \in \descendinginterval{n - 2}{0}} (S_G)_{n - 1} = (x_i)_{i \in \descendinginterval{n - 2}{0}},\]
where
\[f_i := \begin{cases} g_{i + 1} \face_k & \text{for } i \in \descendinginterval{n - 2}{k}, \\ (g_k \face_k) g_{k - 1} & \text{for } i = k - 1, \\ g_i & \text{for } i \in \descendinginterval{k - 2}{0} \end{cases}\]
and
\[x_i := \prod_{j \in \ascendinginterval{i + 1}{n - 2}} (x_j^{- 1} \face_{\descendinginterval{j}{i + 1}} \degeneracy_{\ascendinginterval{i}{j - 1}}) \prod_{j \in \descendinginterval{n - 2}{i}} (f_j \face_{\descendinginterval{j}{i + 1}} \degeneracy_{\ascendinginterval{i}{n - 2}}) \text{ for each } i \in \descendinginterval{n - 2}{0}.\]
On the other hand, we get
\[(g_i)_{i \in \descendinginterval{n - 1}{0}} (S_G)_n \face_k = (y_i)_{i \in \descendinginterval{n - 1}{0}} \face_k = (x_i')_{i \in \descendinginterval{n - 2}{0}}\]
with
\[y_i := \prod_{j \in \ascendinginterval{i + 1}{n - 1}} (y_j^{- 1} \face_{\descendinginterval{j}{i + 1}} \degeneracy_{\ascendinginterval{i}{j - 1}}) \prod_{j \in \descendinginterval{n - 1}{i}} (g_j \face_{\descendinginterval{j}{i + 1}} \degeneracy_{\ascendinginterval{i}{n - 1}}) \text{ for } i \in \descendinginterval{n - 1}{0}\]
and
\[x_i' := \begin{cases} y_{i + 1} \face_k & \text{for } i \in \descendinginterval{n - 2}{k}, \\ (y_k \face_k) (y_{k - 1} \face_k) & \text{for } i = k - 1, \\ y_i \face_k & \text{for } i \in \descendinginterval{k - 2}{0}. \end{cases}\]
We have to show that \(x_i = x_i'\) for all \(i \in \descendinginterval{n - 2}{0}\). To this end, we proceed by induction on \(i\).

For \(i \in \descendinginterval{n - 2}{k}\), we calculate
\begin{align*}
x_i & = \prod_{j \in \ascendinginterval{i + 1}{n - 2}} (x_j^{- 1} \face_{\descendinginterval{j}{i + 1}} \degeneracy_{\ascendinginterval{i}{j - 1}}) \prod_{j \in \descendinginterval{n - 2}{i}} (f_j \face_{\descendinginterval{j}{i + 1}} \degeneracy_{\ascendinginterval{i}{n - 2}}) \\
& = \prod_{j \in \ascendinginterval{i + 1}{n - 2}} ({x_j'}^{- 1} \face_{\descendinginterval{j}{i + 1}} \degeneracy_{\ascendinginterval{i}{j - 1}}) \prod_{j \in \descendinginterval{n - 2}{i}} (f_j \face_{\descendinginterval{j}{i + 1}} \degeneracy_{\ascendinginterval{i}{n - 2}}) \\
& = \prod_{j \in \ascendinginterval{i + 1}{n - 2}} (y_{j + 1}^{- 1} \face_k \face_{\descendinginterval{j}{i + 1}} \degeneracy_{\ascendinginterval{i}{j - 1}}) \prod_{j \in \descendinginterval{n - 2}{i}} (g_{j + 1} \face_k \face_{\descendinginterval{j}{i + 1}} \degeneracy_{\ascendinginterval{i}{n - 2}}) \\
& = \Big( \prod_{j \in \ascendinginterval{i + 2}{n - 1}} (y_j^{- 1} \face_{\descendinginterval{j}{i + 2}} \degeneracy_{\ascendinginterval{i + 1}{j - 1}}) \prod_{j \in \descendinginterval{n - 1}{i + 1}} (g_j \face_{\descendinginterval{j}{i + 2}} \degeneracy_{\ascendinginterval{i + 1}{n - 1}}) \Big) \face_k = y_{i + 1} \face_k.
\end{align*}
For \(i = k - 1\), we have
\begin{align*}
x_{k - 1} & = \prod_{j \in \ascendinginterval{k}{n - 2}} (x_j^{- 1} \face_{\descendinginterval{j}{k}} \degeneracy_{\ascendinginterval{k - 1}{j - 1}}) \prod_{j \in \descendinginterval{n - 2}{k - 1}} (f_j \face_{\descendinginterval{j}{k}} \degeneracy_{\ascendinginterval{k - 1}{n - 2}}) \\
& = \prod_{j \in \ascendinginterval{k}{n - 2}} ({x_j'}^{- 1} \face_{\descendinginterval{j}{k}} \degeneracy_{\ascendinginterval{k - 1}{j - 1}}) \prod_{j \in \descendinginterval{n - 2}{k - 1}} (f_j \face_{\descendinginterval{j}{k}} \degeneracy_{\ascendinginterval{k - 1}{n - 2}}) \\
& = \prod_{j \in \ascendinginterval{k}{n - 2}} (y_{j + 1}^{- 1} \face_k \face_{\descendinginterval{j}{k}} \degeneracy_{\ascendinginterval{k - 1}{j - 1}}) \prod_{j \in \descendinginterval{n - 2}{k}} (g_{j + 1} \face_k \face_{\descendinginterval{j}{k}} \degeneracy_{\ascendinginterval{k - 1}{n - 2}}) \cdot ((g_k \face_k) g_{k - 1}) \degeneracy_{\ascendinginterval{k - 1}{n - 2}} \\
& = \prod_{j \in \ascendinginterval{k + 1}{n - 1}} (y_j^{- 1} \face_{\descendinginterval{j}{k}} \degeneracy_{\ascendinginterval{k - 1}{j - 2}}) \prod_{j \in \descendinginterval{n - 1}{k - 1}} (g_j \face_{\descendinginterval{j}{k}} \degeneracy_{\ascendinginterval{k - 1}{n - 2}}) \\
& = (y_k \face_k) \prod_{j \in \ascendinginterval{k}{n - 1}} (y_j^{- 1} \face_{\descendinginterval{j}{k}} \degeneracy_{\ascendinginterval{k - 1}{j - 2}}) \prod_{j \in \descendinginterval{n - 1}{k - 1}} (g_j \face_{\descendinginterval{j}{k}} \degeneracy_{\ascendinginterval{k - 1}{n - 2}}) \\
& = (y_k \face_k) \Big( \prod_{j \in \ascendinginterval{k}{n - 1}} (y_j^{- 1} \face_{\descendinginterval{j}{k}} \degeneracy_{\ascendinginterval{k - 1}{j - 1}}) \prod_{j \in \descendinginterval{n - 1}{k - 1}} (g_j \face_{\descendinginterval{j}{k}} \degeneracy_{\ascendinginterval{k - 1}{n - 1}}) \Big) \face_k \\
& = (y_k \face_k) (y_{k - 1} \face_k).
\end{align*}
For \(i \in \descendinginterval{k - 2}{0}\), we finally get
\begin{align*}
x_i & = \prod_{j \in \ascendinginterval{i + 1}{n - 2}} (x_j^{- 1} \face_{\descendinginterval{j}{i + 1}} \degeneracy_{\ascendinginterval{i}{j - 1}}) \prod_{j \in \descendinginterval{n - 2}{i}} (f_j \face_{\descendinginterval{j}{i + 1}} \degeneracy_{\ascendinginterval{i}{n - 2}}) \\
& = \prod_{j \in \ascendinginterval{i + 1}{n - 2}} ({x_j'}^{- 1} \face_{\descendinginterval{j}{i + 1}} \degeneracy_{\ascendinginterval{i}{j - 1}}) \prod_{j \in \descendinginterval{n - 2}{i}} (f_j \face_{\descendinginterval{j}{i + 1}} \degeneracy_{\ascendinginterval{i}{n - 2}}) \\
& = (\prod_{j \in \ascendinginterval{i + 1}{k - 2}} (y_j^{- 1} \face_k \face_{\descendinginterval{j}{i + 1}} \degeneracy_{\ascendinginterval{i}{j - 1}})) ((y_k y_{k - 1})^{- 1} \face_k \face_{\descendinginterval{k - 1}{i + 1}} \degeneracy_{\ascendinginterval{i}{k - 2}}) \\
& \qquad \cdot (\prod_{j \in \ascendinginterval{k}{n - 2}} (y_{j + 1}^{- 1} \face_k \face_{\descendinginterval{j}{i + 1}} \degeneracy_{\ascendinginterval{i}{j - 1}})) (\prod_{j \in \descendinginterval{n - 2}{k}} (y_{j + 1} \face_k \face_{\descendinginterval{j}{i + 1}} \degeneracy_{\ascendinginterval{i}{n - 2}})) \\
& \qquad \cdot (((g_k \face_k) g_{k - 1}) \face_{\descendinginterval{k - 1}{i + 1}} \degeneracy_{\ascendinginterval{i}{n - 2}}) (\prod_{j \in \descendinginterval{k - 2}{i}} (g_j \face_{\descendinginterval{j}{i + 1}} \degeneracy_{\ascendinginterval{i}{n - 2}})) \\
& = \prod_{j \in \ascendinginterval{i + 1}{k - 1}} (y_j^{- 1} \face_k \face_{\descendinginterval{j}{i + 1}} \degeneracy_{\ascendinginterval{i}{j - 1}}) \prod_{j \in \ascendinginterval{k}{n - 1}} (y_j^{- 1} \face_{\descendinginterval{j}{i + 1}} \degeneracy_{\ascendinginterval{i}{j - 2}}) \\
& \qquad \cdot \prod_{j \in \descendinginterval{n - 1}{k}} (g_j \face_{\descendinginterval{j}{i + 1}} \degeneracy_{\ascendinginterval{i}{n - 2}}) \prod_{j \in \descendinginterval{k - 1}{i}} (g_j \face_{\descendinginterval{j}{i + 1}} \degeneracy_{\ascendinginterval{i}{n - 2}}) \\
& = \Big( \prod_{j \in \ascendinginterval{i + 1}{n - 1}} (y_j^{- 1} \face_{\descendinginterval{j}{i + 1}} \degeneracy_{\ascendinginterval{i}{j - 1}}) \prod_{j \in \descendinginterval{n - 1}{i}} (g_j \face_{\descendinginterval{j}{i + 1}} \degeneracy_{\ascendinginterval{i}{n - 1}}) \Big) \face_k = y_i \face_k.
\end{align*}

Next, we come to the degeneracies. We let \(n \in \Naturals_0\), \(k \in [0, n]\) and \((g_i)_{i \in \descendinginterval{n - 1}{0}} \in \KanClassifyingSimplicialSet_n G\). Then we have
\[(g_i)_{i \in \descendinginterval{n - 1}{0}} \degeneracy_k (S_G)_{n + 1} = (h_i)_{i \in \descendinginterval{n}{0}} (S_G)_{n + 1} = (z_i)_{i \in \descendinginterval{n}{0}},\]
where
\[h_i := \begin{cases} g_{i - 1} \degeneracy_k & \text{for } i \in \descendinginterval{n}{k + 1}, \\ 1 & \text{for } i = k, \\ g_i & \text{for } i \in \descendinginterval{k - 1}{0} \end{cases}\]
and
\[z_i := \prod_{j \in \ascendinginterval{i + 1}{n}} (z_j^{- 1} \face_{\descendinginterval{j}{i + 1}} \degeneracy_{\ascendinginterval{i}{j - 1}}) \prod_{j \in \descendinginterval{n}{i}} (h_j \face_{\descendinginterval{j}{i + 1}} \degeneracy_{\ascendinginterval{i}{n}}) \text{ for each } i \in \descendinginterval{n}{0}.\]
Further, we get
\[(g_i)_{i \in \descendinginterval{n - 1}{0}} (S_G)_n \degeneracy_k = (y_i)_{i \in \descendinginterval{n - 1}{0}} \degeneracy_k = (z_i')_{i \in \descendinginterval{n}{0}}\]
with
\[y_i := \prod_{j \in \ascendinginterval{i + 1}{n - 1}} (y_j^{- 1} \face_{\descendinginterval{j}{i + 1}} \degeneracy_{\ascendinginterval{i}{j - 1}}) \prod_{j \in \descendinginterval{n - 1}{i}} (g_j \face_{\descendinginterval{j}{i + 1}} \degeneracy_{\ascendinginterval{i}{n - 1}}) \text{ for } i \in \descendinginterval{n - 1}{0}\]
and
\[z_i' := \begin{cases} y_{i - 1} \degeneracy_k & \text{for } i \in \descendinginterval{n}{k + 1}, \\ 1 & \text{for } i = k, \\ y_i \degeneracy_k & \text{for } i \in \descendinginterval{k - 1}{0}. \end{cases}\]
Thus we have to show that \(z_i = z_i'\) for every \(i \in \descendinginterval{n}{0}\). To this end, we perform an induction on \(i \in \descendinginterval{n}{0}\).

For \(i \in \descendinginterval{n}{k + 1}\), we have
\begin{align*}
z_i & = \prod_{j \in \ascendinginterval{i + 1}{n}} (z_j^{- 1} \face_{\descendinginterval{j}{i + 1}} \degeneracy_{\ascendinginterval{i}{j - 1}}) \prod_{j \in \descendinginterval{n}{i}} (h_j \face_{\descendinginterval{j}{i + 1}} \degeneracy_{\ascendinginterval{i}{n}}) \\
& = \prod_{j \in \ascendinginterval{i + 1}{n}} ({z_j'}^{- 1} \face_{\descendinginterval{j}{i + 1}} \degeneracy_{\ascendinginterval{i}{j - 1}}) \prod_{j \in \descendinginterval{n}{i}} (h_j \face_{\descendinginterval{j}{i + 1}} \degeneracy_{\ascendinginterval{i}{n}}) \\
& = \prod_{j \in \ascendinginterval{i + 1}{n}} (y_{j - 1}^{- 1} \degeneracy_k \face_{\descendinginterval{j}{i + 1}} \degeneracy_{\ascendinginterval{i}{j - 1}}) \prod_{j \in \descendinginterval{n}{i}} (g_{j - 1} \degeneracy_k \face_{\descendinginterval{j}{i + 1}} \degeneracy_{\ascendinginterval{i}{n}}) \\
& = \Big(\prod_{j \in \ascendinginterval{i}{n - 1}} (y_j^{- 1} \face_{\descendinginterval{j}{i}} \degeneracy_{\ascendinginterval{i - 1}{j - 1}}) \prod_{j \in \descendinginterval{n - 1}{i - 1}} (g_j \face_{\descendinginterval{j}{i}} \degeneracy_{\ascendinginterval{i - 1}{n - 1}}) \Big) \degeneracy_k = y_{i - 1} \degeneracy_k.
\end{align*}
For \(i = k\), we compute
\begin{align*}
z_k & = \prod_{j \in \ascendinginterval{k + 1}{n}} (z_j^{- 1} \face_{\descendinginterval{j}{k + 1}} \degeneracy_{\ascendinginterval{k}{j - 1}}) \prod_{j \in \descendinginterval{n}{k}} (h_j \face_{\descendinginterval{j}{k + 1}} \degeneracy_{\ascendinginterval{k}{n}}) \\
& = \prod_{j \in \ascendinginterval{k + 1}{n}} ({z_j'}^{- 1} \face_{\descendinginterval{j}{k + 1}} \degeneracy_{\ascendinginterval{k}{j - 1}}) \prod_{j \in \descendinginterval{n}{k}} (h_j \face_{\descendinginterval{j}{k + 1}} \degeneracy_{\ascendinginterval{k}{n}}) \\
& = \prod_{j \in \ascendinginterval{k + 1}{n}} (y_{j - 1}^{- 1} \degeneracy_k \face_{\descendinginterval{j}{k + 1}} \degeneracy_{\ascendinginterval{k}{j - 1}}) \prod_{j \in \descendinginterval{n}{k + 1}} (g_{j - 1} \degeneracy_k \face_{\descendinginterval{j}{k + 1}} \degeneracy_{\ascendinginterval{k}{n}}) \\
& = \prod_{j \in \ascendinginterval{k + 1}{n}} (y_{j - 1}^{- 1} \face_{\descendinginterval{j - 1}{k + 1}} \degeneracy_{\ascendinginterval{k}{j - 1}}) \prod_{j \in \descendinginterval{n}{k + 1}} (g_{j - 1} \face_{\descendinginterval{j - 1}{k + 1}} \degeneracy_{\ascendinginterval{k}{n}}) \\
& = \prod_{j \in \ascendinginterval{k + 1}{n}} (y_{j - 1}^{- 1} \degeneracy_k \face_{\descendinginterval{j}{k + 2}} \degeneracy_{\ascendinginterval{k + 1}{j - 1}}) \prod_{j \in \descendinginterval{n}{k + 1}} (g_{j - 1} \degeneracy_k \face_{\descendinginterval{j}{k + 2}} \degeneracy_{\ascendinginterval{k + 1}{n}}) \\
& = \prod_{j \in \ascendinginterval{k + 1}{n}} ({z_j'}^{- 1} \face_{\descendinginterval{j}{k + 2}} \degeneracy_{\ascendinginterval{k + 1}{j - 1}}) \prod_{j \in \descendinginterval{n}{k + 1}} (h_j \face_{\descendinginterval{j}{k + 2}} \degeneracy_{\ascendinginterval{k + 1}{n}}) \\
& = z_{k + 1}^{- 1} \prod_{j \in \ascendinginterval{k + 2}{n}} (z_j^{- 1} \face_{\descendinginterval{j}{k + 2}} \degeneracy_{\ascendinginterval{k + 1}{j - 1}}) \prod_{j \in \descendinginterval{n}{k + 1}} (h_j \face_{\descendinginterval{j}{k + 2}} \degeneracy_{\ascendinginterval{k + 1}{n}}) = z_{k + 1}^{- 1} z_{k + 1} = 1.
\end{align*}
For \(i \in \descendinginterval{k - 1}{0}\), we get
\begin{align*}
z_i & = \prod_{j \in \ascendinginterval{i + 1}{n}} (z_j^{- 1} \face_{\descendinginterval{j}{i + 1}} \degeneracy_{\ascendinginterval{i}{j - 1}}) \prod_{j \in \descendinginterval{n}{i}} (h_j \face_{\descendinginterval{j}{i + 1}} \degeneracy_{\ascendinginterval{i}{n}}) \\
& = \prod_{j \in \ascendinginterval{i + 1}{n}} ({z_j'}^{- 1} \face_{\descendinginterval{j}{i + 1}} \degeneracy_{\ascendinginterval{i}{j - 1}}) \prod_{j \in \descendinginterval{n}{i}} (h_j \face_{\descendinginterval{j}{i + 1}} \degeneracy_{\ascendinginterval{i}{n}}) \\
& = \prod_{j \in \ascendinginterval{i + 1}{k - 1}} (y_j^{- 1} \degeneracy_k \face_{\descendinginterval{j}{i + 1}} \degeneracy_{\ascendinginterval{i}{j - 1}}) \prod_{j \in \ascendinginterval{k + 1}{n}} (y_{j - 1}^{- 1} \degeneracy_k \face_{\descendinginterval{j}{i + 1}} \degeneracy_{\ascendinginterval{i}{j - 1}}) \\
& \qquad \cdot \prod_{j \in \descendinginterval{n}{k + 1}} (g_{j - 1} \degeneracy_k \face_{\descendinginterval{j}{i + 1}} \degeneracy_{\ascendinginterval{i}{n}}) \prod_{j \in \descendinginterval{k - 1}{i}} (g_j \face_{\descendinginterval{j}{i + 1}} \degeneracy_{\ascendinginterval{i}{n}}) \\
& = \prod_{j \in \ascendinginterval{i + 1}{k - 1}} (y_j^{- 1} \degeneracy_k \face_{\descendinginterval{j}{i + 1}} \degeneracy_{\ascendinginterval{i}{j - 1}}) \prod_{j \in \ascendinginterval{k}{n - 1}} (y_j^{- 1} \degeneracy_k \face_{\descendinginterval{j + 1}{i + 1}} \degeneracy_{\ascendinginterval{i}{j}}) \\
& \qquad \cdot \prod_{j \in \descendinginterval{n - 1}{k}} (g_j \degeneracy_k \face_{\descendinginterval{j + 1}{i + 1}} \degeneracy_{\ascendinginterval{i}{n}}) \prod_{j \in \descendinginterval{k - 1}{i}} (g_j \face_{\descendinginterval{j}{i + 1}} \degeneracy_{\ascendinginterval{i}{n}}) \\
& = \Big( \prod_{j \in \ascendinginterval{i + 1}{n - 1}} (y_j^{- 1} \face_{\descendinginterval{j}{i + 1}} \degeneracy_{\ascendinginterval{i}{j - 1}}) \prod_{j \in \descendinginterval{n - 1}{i}} (g_j \face_{\descendinginterval{j}{i + 1}} \degeneracy_{\ascendinginterval{i}{n - 1}}) \Big) \degeneracy_k = y_i \degeneracy_k.
\end{align*}
	
Thus \((S_G)_{n \in \Naturals}\) yields a simplicial map
\[\KanClassifyingSimplicialSet G \morphism[S_G] \DiagonalFunctor \Nerve G.\]

Finally, we have to prove that \(D_G\) is a retraction with coretraction \(S_G\), that is,
\[(S_G)_n (D_G)_n = \id_{\KanClassifyingSimplicialSet_n G} \text{ for all } n \in \Naturals_0.\]
Again, we let \((y_i)_{i \in \descendinginterval{n - 1}{0}}\) denote the image of an element \((g_i)_{i \in \descendinginterval{n - 1}{0}} \in \KanClassifyingSimplicialSet_n G\) under \((S_G)_n\). Then we have
\[(g_i)_{i \in \descendinginterval{n - 1}{0}} (S_G)_n (D_G)_n = (y_i)_{i \in \descendinginterval{n - 1}{0}} (D_G)_n = (y_i \face_{\descendinginterval{n}{i + 1}})_{i \in \descendinginterval{n - 1}{0}}.\]
Induction on \(i \in \descendinginterval{n - 1}{0}\) shows that
\begin{align*}
y_i \face_{\descendinginterval{n}{i + 1}} 
& = \prod_{j \in \ascendinginterval{i + 1}{n - 1}} (y_j^{- 1} \face_{\descendinginterval{j}{i + 1}} \degeneracy_{\ascendinginterval{i}{j - 1}} \face_{\descendinginterval{n}{i + 1}}) \prod_{j \in \descendinginterval{n - 1}{i}} (g_j \face_{\descendinginterval{j}{i + 1}} \degeneracy_{\ascendinginterval{i}{n - 1}} \face_{\descendinginterval{n}{i + 1}}) \\
& = \prod_{j \in \ascendinginterval{i + 1}{n - 1}} (y_j^{- 1} \face_{\descendinginterval{n}{i + 1}}) \prod_{j \in \descendinginterval{n - 1}{i}} (g_j \face_{\descendinginterval{j}{i + 1}}) \\
& = \prod_{j \in \ascendinginterval{i + 1}{n - 1}} (g_j^{- 1} \face_{\descendinginterval{j}{i + 1}}) \prod_{j \in \descendinginterval{n - 1}{i}} (g_j \face_{\descendinginterval{j}{i + 1}}) = g_i.
\end{align*}
This implies that \((S_G)_n (D_G)_n = \id_{\KanClassifyingSimplicialSet_n G}\) for all \(n \in \Naturals_0\).
\end{proof}

\begin{theorem*}
We suppose given a simplicial group \(G\). The Kan classifying simplicial set \(\KanClassifyingSimplicialSet G\) is a strong simplicial deformation retract of \(\DiagonalFunctor \Nerve G\) with a strong simplicial deformation retraction given by
\[\DiagonalFunctor \Nerve G \morphism[D_G] \KanClassifyingSimplicialSet G.\]
\end{theorem*}
\begin{proof}
We consider the coretraction \(\KanClassifyingSimplicialSet \morphism[S] \DiagonalFunctor \Nerve\) as in the preceding proposition. Now, we shall show that \(D_G S_G \homotopic \id_{\DiagonalFunctor \Nerve G}\) via a simplicial homotopy constant along \(S_G\).

A simplicial homotopy \(H\) from \(D_G S_G\) to \(\id_{\DiagonalFunctor \Nerve G}\) is given by
\[H_n\colon \DiagonalFunctor_n \Nerve G \times \StandardSimplex{1}_n \map \DiagonalFunctor_n \Nerve G, ((g_{n, i})_{i \in \descendinginterval{n - 1}{0}}, \cotime^{n + 1 - k}) \mapsto (y_i^{(n + 1 - k)})_{i \in \descendinginterval{n - 1}{0}}\]
for all \(n \in \Naturals_0\), where \(k \in [0, n + 1]\) and, defined by descending recursion,
\[y_i^{(n + 1 - k)} := \begin{cases} g_{n, i} & \text{for } i \in \descendinginterval{n - 1}{k - 1} \cap \Naturals_0, \\ \prod_{j \in \ascendinginterval{i + 1}{k - 2}} ((y_j^{(n + 1 - k)})^{- 1} \face_{\descendinginterval{j}{i + 1}} \degeneracy_{\ascendinginterval{i}{j - 1}}) \\ \qquad \cdot \prod_{j \in \descendinginterval{k - 2}{i}} (g_{n, j} \face_{\descendinginterval{k - 1}{i + 1}} \degeneracy_{\ascendinginterval{i}{k - 2}}) & \text{for } i \in \descendinginterval{k - 2}{0}. \end{cases}\]

To facilitate the following calculations, we abbreviate \(\tilde y_i := y_i^{(n + 1 - k)}\) for the respective index \(k \in [0, n]\) under consideration, if no confusion can arise.


We have to verify that the maps \(H_n\) for \(n \in \Naturals_0\) yield a simplicial map.

First, we show the compatibility with the faces. For \(k \in [0, n]\), \(l \in [0, n + 1]\), \(n \in \Naturals_0\), \((g_{n, i})_{i \in \descendinginterval{n - 1}{0}} \in \DiagonalFunctor_n \Nerve G\), we have
\begin{align*}
((g_{n, i})_{i \in \descendinginterval{n - 1}{0}}, \cotime^{n + 1 - l}) \face_k H_{n - 1} & = ((g_{n, i})_{i \in \descendinginterval{n - 1}{0}} \face_k, \cotime^{n + 1 - l} \face_k) H_{n - 1} = ((f_i)_{i \in \descendinginterval{n - 2}{0}}, \coface^k \cotime^{n + 1 - l}) H_{n - 1} \\
& = \left. \begin{cases} ((f_i)_{i \in \descendinginterval{n - 2}{0}}, \cotime^{n - l}) H_{n - 1} & \text{for } k \geq l, \\ ((f_i)_{i \in \descendinginterval{n - 2}{0}}, \cotime^{n + 1 - l}) H_{n - 1} & \text{for } k < l \end{cases} \right\} = (\tilde x_i)_{i \in \descendinginterval{n - 2}{0}},
\end{align*}
where
\[f_i := \begin{cases} g_{n, i + 1} \face_k & \text{for } i \in \descendinginterval{n - 2}{k}, \\ (g_{n, k} \face_k) (g_{n, k - 1} \face_k) & \text{for } i = k - 1, \\ g_{n, i} \face_k & \text{for } i \in \descendinginterval{k - 2}{0} \end{cases}\]
for all \(i \in \descendinginterval{n - 2}{0}\) and
\[\tilde x_i := \begin{cases} \left. \begin{cases} f_i & \text{for } i \in \descendinginterval{n - 2}{l - 1}, \\ \prod_{j \in \ascendinginterval{i + 1}{l - 2}} (\tilde x_j^{- 1} \face_{\descendinginterval{j}{i + 1}} \degeneracy_{\ascendinginterval{i}{j - 1}}) & \\ \qquad \cdot \prod_{j \in \descendinginterval{l - 2}{i}} (f_j \face_{\descendinginterval{l - 1}{i + 1}} \degeneracy_{\ascendinginterval{i}{l - 2}}) & \text{for } i \in \descendinginterval{l - 2}{0} \end{cases} \right\} & \text{if } k \geq l, \\ \left. \begin{cases} f_i & \text{for } i \in \descendinginterval{n - 2}{l - 2}, \\ \prod_{j \in \ascendinginterval{i + 1}{l - 3}} (\tilde x_j^{- 1} \face_{\descendinginterval{j}{i + 1}} \degeneracy_{\ascendinginterval{i}{j - 1}}) & \\ \qquad \cdot \prod_{j \in \descendinginterval{l - 3}{i}} (f_j \face_{\descendinginterval{l - 2}{i + 1}} \degeneracy_{\ascendinginterval{i}{l - 3}}) & \text{for } i \in \descendinginterval{l - 3}{0} \end{cases} \right\} & \text{if } k < l \end{cases}\]
for all \(i \in \descendinginterval{n - 2}{0}\). On the other hand, we have
\[((g_{n, i})_{i \in \descendinginterval{n - 1}{0}}, \cotime^{n + 1 - l}) H_n \face_k = (\tilde y_i)_{i \in \descendinginterval{n - 1}{0}} \face_k = (\tilde x_i')_{i \in \descendinginterval{n - 2}{0}}\]
with
\[\tilde y_i := \begin{cases} g_{n, i} & \text{for } i \in \descendinginterval{n - 1}{l - 1}, \\ \prod_{j \in \ascendinginterval{i + 1}{l - 2}} (\tilde y_j^{- 1} \face_{\descendinginterval{j}{i + 1}} \degeneracy_{\ascendinginterval{i}{j - 1}}) \prod_{j \in \descendinginterval{l - 2}{i}} (g_{n, j} \face_{\descendinginterval{l - 1}{i + 1}} \degeneracy_{\ascendinginterval{i}{l - 2}}) & \text{for } i \in \descendinginterval{l - 2}{0} \end{cases}\]
for \(i \in \descendinginterval{n - 1}{0}\) and
\[\tilde x_i' := \begin{cases} \tilde y_{i + 1} \face_k & \text{for } i \in \descendinginterval{n - 2}{k}, \\ (\tilde y_k \face_k) (\tilde y_{k - 1} \face_k) & \text{for } i = k - 1, \\ \tilde y_i \face_k & \text{for } i \in \descendinginterval{k - 2}{0} \end{cases}\]
for \(i \in \descendinginterval{n - 2}{0}\). We have to show that \(\tilde x_i = \tilde x_i'\) for all \(i \in \descendinginterval{n - 2}{0}\). To this end, we consider three cases and we handle each one by induction on \(i \in \descendinginterval{n - 2}{0}\).

We suppose that \(k \in \descendinginterval{n}{l}\). For \(i \in \descendinginterval{n - 2}{k}\), we have
\[\tilde x_i = f_i = g_{n, i + 1} \face_k = \tilde y_{i + 1} \face_k = \tilde x_i'.\]
For \(i = k - 1\), we get
\[\tilde x_{k - 1} = f_{k - 1} = (g_{n, k} \face_k) (g_{n, k - 1} \face_k) = (\tilde y_k \face_k) (\tilde y_{k - 1} \face_k) = \tilde x_{k - 1}'.\]
For \(i \in \descendinginterval{k - 2}{l - 1}\), we get
\[\tilde x_i = f_i = g_{n, i} \face_k = \tilde y_i \face_k = \tilde x_i'\]
Finally, for \(i \in \descendinginterval{l - 2}{0}\), we calculate
\begin{align*}
\tilde x_i & = \prod_{j \in \ascendinginterval{i + 1}{l - 2}} (\tilde x_j^{- 1} \face_{\descendinginterval{j}{i + 1}} \degeneracy_{\ascendinginterval{i}{j - 1}}) \prod_{j \in \descendinginterval{l - 2}{i}} (f_j \face_{\descendinginterval{l - 1}{i + 1}} \degeneracy_{\ascendinginterval{i}{l - 2}}) \\
& = \prod_{j \in \ascendinginterval{i + 1}{l - 2}} ({\tilde x_j'}{}^{- 1} \face_{\descendinginterval{j}{i + 1}} \degeneracy_{\ascendinginterval{i}{j - 1}}) \prod_{j \in \descendinginterval{l - 2}{i}} (f_j \face_{\descendinginterval{l - 1}{i + 1}} \degeneracy_{\ascendinginterval{i}{l - 2}}) \\
& = \prod_{j \in \ascendinginterval{i + 1}{l - 2}} (\tilde y_j^{- 1} \face_k \face_{\descendinginterval{j}{i + 1}} \degeneracy_{\ascendinginterval{i}{j - 1}}) \prod_{j \in \descendinginterval{l - 2}{i}} (g_{n, j} \face_k \face_{\descendinginterval{l - 1}{i + 1}} \degeneracy_{\ascendinginterval{i}{l - 2}}) \\
& = \Big( \prod_{j \in \ascendinginterval{i + 1}{l - 2}} (\tilde y_j^{- 1} \face_{\descendinginterval{j}{i + 1}} \degeneracy_{\ascendinginterval{i}{j - 1}}) \prod_{j \in \descendinginterval{l - 2}{i}} (g_{n, j} \face_{\descendinginterval{l - 1}{i + 1}} \degeneracy_{\ascendinginterval{i}{l - 2}}) \Big) \face_k = \tilde y_i \face_k = \tilde x_i'.
\end{align*}

Next, we suppose that \(k = l - 1\). For \(i \in \descendinginterval{n - 2}{k}\), we have
\[\tilde x_i = f_i = g_{n, i + 1} \face_k = \tilde y_{i + 1} \face_k = \tilde x_i'.\]
For \(i = k - 1\), we compute
\[\tilde x_{k - 1} = f_{k - 1} = (g_{n, k} \face_k) (g_{n, k - 1} \face_k) = (g_{n, k} \face_k) (g_{n, k - 1} \face_k \degeneracy_{k - 1} \face_k) = (\tilde y_k \face_k) (\tilde y_{k - 1} \face_k) = \tilde x_{k - 1}'.\]
For \(i \in \descendinginterval{k - 2}{0}\), we get
\begin{align*}
\tilde x_i & = \prod_{j \in \ascendinginterval{i + 1}{k - 2}} (\tilde x_j^{- 1} \face_{\descendinginterval{j}{i + 1}} \degeneracy_{\ascendinginterval{i}{j - 1}}) \prod_{j \in \descendinginterval{k - 2}{i}} (f_j \face_{\descendinginterval{k - 1}{i + 1}} \degeneracy_{\ascendinginterval{i}{k - 2}}) \\
& = \prod_{j \in \ascendinginterval{i + 1}{k - 2}} ({\tilde x_j'}{}^{- 1} \face_{\descendinginterval{j}{i + 1}} \degeneracy_{\ascendinginterval{i}{j - 1}}) \prod_{j \in \descendinginterval{k - 2}{i}} (f_j \face_{\descendinginterval{k - 1}{i + 1}} \degeneracy_{\ascendinginterval{i}{k - 2}}) \\
& = \prod_{j \in \ascendinginterval{i + 1}{k - 2}} (\tilde y_j^{- 1} \face_k \face_{\descendinginterval{j}{i + 1}} \degeneracy_{\ascendinginterval{i}{j - 1}}) \prod_{j \in \descendinginterval{k - 2}{i}} (g_{n, j} \face_k \face_{\descendinginterval{k - 1}{i + 1}} \degeneracy_{\ascendinginterval{i}{k - 2}}) \\
& = \Big( \prod_{j \in \ascendinginterval{i + 1}{k - 1}} (\tilde y_j^{- 1} \face_{\descendinginterval{j}{i + 1}} \degeneracy_{\ascendinginterval{i}{j - 1}}) \prod_{j \in \descendinginterval{k - 1}{i}} (g_{n, j} \face_{\descendinginterval{k}{i + 1}} \degeneracy_{\ascendinginterval{i}{k - 1}}) \Big) \face_k = \tilde y_i \face_k = \tilde x_i'.
\end{align*}

Finally, we suppose that \(k \in \descendinginterval{l - 2}{0}\). For \(i \in \descendinginterval{n - 2}{l - 2}\), we see that
\[\tilde x_i = f_i = g_{n, i + 1} \face_k = \tilde y_{i + 1} \face_k = \tilde x_i'.\]
For \(i \in \descendinginterval{l - 3}{k}\), we have
\begin{align*}
\tilde x_i & = \prod_{j \in \ascendinginterval{i + 1}{l - 3}} (\tilde x_j^{- 1} \face_{\descendinginterval{j}{i + 1}} \degeneracy_{\ascendinginterval{i}{j - 1}}) \prod_{j \in \descendinginterval{l - 3}{i}} (f_j \face_{\descendinginterval{l - 2}{i + 1}} \degeneracy_{\ascendinginterval{i}{l - 3}}) \\
& = \prod_{j \in \ascendinginterval{i + 1}{l - 3}} ({\tilde x_j'}{}^{- 1} \face_{\descendinginterval{j}{i + 1}} \degeneracy_{\ascendinginterval{i}{j - 1}}) \prod_{j \in \descendinginterval{l - 3}{i}} (f_j \face_{\descendinginterval{l - 2}{i + 1}} \degeneracy_{\ascendinginterval{i}{l - 3}}) \\
& = \prod_{j \in \ascendinginterval{i + 1}{l - 3}} (\tilde y_{j + 1}^{- 1} \face_k \face_{\descendinginterval{j}{i + 1}} \degeneracy_{\ascendinginterval{i}{j - 1}}) \prod_{j \in \descendinginterval{l - 3}{i}} (g_{n, j + 1} \face_k \face_{\descendinginterval{l - 2}{i + 1}} \degeneracy_{\ascendinginterval{i}{l - 3}}) \\
& = \Big( \prod_{j \in \ascendinginterval{i + 2}{l - 2}} (\tilde y_j^{- 1} \face_{\descendinginterval{j}{i + 2}} \degeneracy_{\ascendinginterval{i + 1}{j - 1}}) \prod_{j \in \descendinginterval{l - 2}{i + 1}} (g_{n, j} \face_{\descendinginterval{l - 1}{i + 2}} \degeneracy_{\ascendinginterval{i + 1}{l - 2}}) \Big) \face_k = \tilde y_{i + 1} \face_k = \tilde x_i'.
\end{align*}
For \(i = k - 1\), we have
\begin{align*}
\tilde x_{k - 1} & = \prod_{j \in \ascendinginterval{k}{l - 3}} (\tilde x_j^{- 1} \face_{\descendinginterval{j}{k}} \degeneracy_{\ascendinginterval{k - 1}{j - 1}}) \prod_{j \in \descendinginterval{l - 3}{k - 1}} (f_j \face_{\descendinginterval{l - 2}{k}} \degeneracy_{\ascendinginterval{k - 1}{l - 3}}) \\
& = \prod_{j \in \ascendinginterval{k}{l - 3}} ({\tilde x_j'}{}^{- 1} \face_{\descendinginterval{j}{k}} \degeneracy_{\ascendinginterval{k - 1}{j - 1}}) \prod_{j \in \descendinginterval{l - 3}{k - 1}} (f_j \face_{\descendinginterval{l - 2}{k}} \degeneracy_{\ascendinginterval{k - 1}{l - 3}}) \\
& = (\prod_{j \in \ascendinginterval{k}{l - 3}} (\tilde y_{j + 1} \face_k \face_{\descendinginterval{j}{k}} \degeneracy_{\ascendinginterval{k - 1}{j - 1}})) (\prod_{j \in \descendinginterval{l - 3}{k}} (g_{n, j + 1} \face_k \face_{\descendinginterval{l - 2}{k}} \degeneracy_{\ascendinginterval{k - 1}{l - 3}})) \\
& \qquad \cdot (g_{n, k} \face_k \face_{\descendinginterval{l - 2}{k}} \degeneracy_{\ascendinginterval{k - 1}{l - 3}}) (g_{n, k - 1} \face_k \face_{\descendinginterval{l - 2}{k}} \degeneracy_{\ascendinginterval{k - 1}{l - 3}}) \\
& = \prod_{j \in \ascendinginterval{k + 1}{l - 2}} (\tilde y_j^{- 1} \face_{\descendinginterval{j}{k}} \degeneracy_{\ascendinginterval{k - 1}{j - 2}}) \prod_{j \in \descendinginterval{l - 2}{k - 1}} (g_{n, j} \face_{\descendinginterval{l - 1}{k}} \degeneracy_{\ascendinginterval{k - 1}{l - 3}}) \\
& = (\tilde y_k \face_k) \Big( \prod_{j \in \ascendinginterval{k}{l - 2}} (\tilde y_j^{- 1} \face_{\descendinginterval{j}{k}} \degeneracy_{\ascendinginterval{k - 1}{j - 1}}) \prod_{j \in \descendinginterval{l - 2}{k - 1}} (g_{n, j} \face_{\descendinginterval{l - 1}{k}} \degeneracy_{\ascendinginterval{k - 1}{l - 2}}) \Big) \face_k = (\tilde y_k \face_k) (\tilde y_{k - 1} \face_k) \\
& = \tilde x_{k - 1}'.
\end{align*}
For \(i \in \descendinginterval{k - 2}{0}\), we get
\begin{align*}
\tilde x_i & = \prod_{j \in \ascendinginterval{i + 1}{l - 3}} (\tilde x_j^{- 1} \face_{\descendinginterval{j}{i + 1}} \degeneracy_{\ascendinginterval{i}{j - 1}}) \prod_{j \in \descendinginterval{l - 3}{i}} (f_j \face_{\descendinginterval{l - 2}{i + 1}} \degeneracy_{\ascendinginterval{i}{l - 3}}) \\
& = \prod_{j \in \ascendinginterval{i + 1}{l - 3}} ({\tilde x_j'}{}^{- 1} \face_{\descendinginterval{j}{i + 1}} \degeneracy_{\ascendinginterval{i}{j - 1}}) \prod_{j \in \descendinginterval{l - 3}{i}} (f_j \face_{\descendinginterval{l - 2}{i + 1}} \degeneracy_{\ascendinginterval{i}{l - 3}}) \\
& = (\prod_{j \in \ascendinginterval{i + 1}{k - 2}} (\tilde y_j^{- 1} \face_k \face_{\descendinginterval{j}{i + 1}} \degeneracy_{\ascendinginterval{i}{j - 1}})) (\tilde y_{k - 1}^{- 1} \face_k \face_{\descendinginterval{k - 1}{i + 1}} \degeneracy_{\ascendinginterval{i}{k - 2}}) (\tilde y_k^{- 1} \face_k \face_{\descendinginterval{k - 1}{i + 1}} \degeneracy_{\ascendinginterval{i}{k - 2}}) \\
& \qquad \cdot (\prod_{j \in \ascendinginterval{k}{l - 3}} (\tilde y_{j + 1}^{- 1} \face_k \face_{\descendinginterval{j}{i + 1}} \degeneracy_{\ascendinginterval{i}{j - 1}})) (\prod_{j \in \descendinginterval{l - 3}{k}} (g_{n, j + 1} \face_k \face_{\descendinginterval{l - 2}{i + 1}} \degeneracy_{\ascendinginterval{i}{l - 3}})) (g_{n, k} \face_k \face_{\descendinginterval{l - 2}{i + 1}} \degeneracy_{\ascendinginterval{i}{l - 3}}) \\
& \qquad \cdot (g_{n, k - 1} \face_k \face_{\descendinginterval{l - 2}{i + 1}} \degeneracy_{\ascendinginterval{i}{l - 3}}) (\prod_{j \in \descendinginterval{k - 2}{i}} (g_{n, j} \face_k \face_{\descendinginterval{l - 2}{i + 1}} \degeneracy_{\ascendinginterval{i}{l - 3}})) \\
& = \prod_{j \in \ascendinginterval{i + 1}{k - 1}} (\tilde y_j^{- 1} \face_k \face_{\descendinginterval{j}{i + 1}} \degeneracy_{\ascendinginterval{i}{j - 1}}) \prod_{j \in \ascendinginterval{k}{l - 2}} (\tilde y_j^{- 1} \face_k \face_{\descendinginterval{j - 1}{i + 1}} \degeneracy_{\ascendinginterval{i}{j - 2}}) \prod_{j \in \descendinginterval{l - 2}{i}} (g_{n, j} \face_k \face_{\descendinginterval{l - 2}{i + 1}} \degeneracy_{\ascendinginterval{i}{l - 3}}) \\
& = \Big( \prod_{j \in \ascendinginterval{i + 1}{l - 2}} (\tilde y_j^{- 1} \face_{\descendinginterval{j}{i + 1}} \degeneracy_{\ascendinginterval{i}{j - 1}}) \prod_{j \in \descendinginterval{l - 2}{i}} (g_{n, j} \face_{\descendinginterval{l - 1}{i + 1}} \degeneracy_{\ascendinginterval{i}{l - 2}}) \Big) \face_k = \tilde y_i \face_k = \tilde x_i'.
\end{align*}

Now we consider the degeneracies. We let \(n \in \Naturals_0\), \(k \in [0, n]\), \(l \in [0, n + 1]\), and \((g_{n, i})_{i \in \descendinginterval{n - 1}{0}} \in \DiagonalFunctor_n \Nerve G\). We compute
\begin{align*}
((g_{n, i})_{i \in \descendinginterval{n - 1}{0}}, \cotime^{n + 1 - l}) \degeneracy_k H_{n + 1} & = ((g_{n, i})_{i \in \descendinginterval{n - 1}{0}} \degeneracy_k, \cotime^{n + 1 - l} \degeneracy_k) H_{n + 1} = ((h_i)_{i \in \descendinginterval{n - 1}{0}}, \codegeneracy^k \cotime^{n + 1 - l}) H_{n + 1} \\
& = \left. \begin{cases} ((h_i)_{i \in \descendinginterval{n - 1}{0}}, \cotime^{n + 2 - l}) H_{n + 1} & \text{for } k \geq l, \\ ((h_i)_{i \in \descendinginterval{n - 1}{0}}, \cotime^{n + 1 - l}) H_{n + 1} & \text{for } k < l \end{cases} \right\} = (\tilde z_i)_{i \in \descendinginterval{n}{0}},
\end{align*}
where
\[h_i := \begin{cases} g_{n, i - 1} \degeneracy_k & \text{for } i \in \descendinginterval{n}{k + 1}, \\ 1 & \text{for } i = k, \\ g_{n, i} \degeneracy_k & \text{for } i \in \descendinginterval{k - 1}{0} \end{cases}\]
and
\[\tilde z_i := \begin{cases} \left. \begin{cases} h_i & \text{for } i \in \descendinginterval{n}{l - 1}, \\ \prod_{j \in \ascendinginterval{i + 1}{l - 2}} (\tilde z_j^{- 1} \face_{\descendinginterval{j}{i + 1}} \degeneracy_{\ascendinginterval{i}{j - 1}}) \prod_{j \in \descendinginterval{l - 2}{i}} (h_j \face_{\descendinginterval{l - 1}{i + 1}} \degeneracy_{\ascendinginterval{i}{l - 2}}) & \text{for } i \in \descendinginterval{l - 2}{0} \end{cases} \right\} & \text{if } k \geq l, \\ \left. \begin{cases} h_i & \text{for } i \in \descendinginterval{n}{l}, \\ \prod_{j \in \ascendinginterval{i + 1}{l - 1}} (\tilde z_j^{- 1} \face_{\descendinginterval{j}{i + 1}} \degeneracy_{\ascendinginterval{i}{j - 1}}) \prod_{j \in \descendinginterval{l - 1}{i}} (h_j \face_{\descendinginterval{l}{i + 1}} \degeneracy_{\ascendinginterval{i}{l - 1}}) & \text{for } i \in \descendinginterval{l - 1}{0} \end{cases} \right\} & \text{if } k < l. \end{cases}\]
Furthermore, we have
\[((g_{n, i})_{i \in \descendinginterval{n - 1}{0}}, \cotime^{n + 1 - l}) H_n \degeneracy_k = (\tilde y_i)_{i \in \descendinginterval{n - 1}{0}} \degeneracy_k = (\tilde z_i')_{i \in \descendinginterval{n}{0}},\]
where
\[\tilde y_i := \begin{cases} g_{n, i} & \text{for } i \in \descendinginterval{n - 1}{l - 1}, \\ \prod_{j \in \ascendinginterval{i + 1}{l - 2}} (\tilde y_j^{- 1} \face_{\descendinginterval{j}{i + 1}} \degeneracy_{\ascendinginterval{i}{j - 1}}) \prod_{j \in \descendinginterval{l - 2}{i}} (g_{n, j} \face_{\descendinginterval{l - 1}{i + 1}} \degeneracy_{\ascendinginterval{i}{l - 2}}) & \text{for } i \in \descendinginterval{l - 2}{0} \end{cases}\]
and
\[\tilde z_i' := \begin{cases} \tilde y_{i - 1} \degeneracy_k & \text{for } i \in \descendinginterval{n}{k + 1}, \\ 1 & \text{for } i = k, \\ \tilde y_i \degeneracy_k & \text{for } i \in \descendinginterval{k - 1}{0}. \end{cases}\]
Thus we have to show that \(\tilde z_i = \tilde z_i'\) for every \(i \in \descendinginterval{n}{0}\). Again, we distinguish three cases, and in each one, we perform an induction on \(i \in \descendinginterval{n}{0}\).

We suppose that \(k \in \descendinginterval{n}{l}\). For \(i \in \descendinginterval{n}{k + 1}\), we calculate
\[\tilde z_i = h_i = g_{n, i - 1} \degeneracy_k = \tilde y_{i - 1} \degeneracy_k = \tilde z_i'.\]
For \(i = k\), we get
\[\tilde z_k = h_k = 1 = \tilde z_k'.\]
For \(i \in \descendinginterval{k - 1}{l - 1}\), we have
\[\tilde z_i = h_i = g_{n, i} \degeneracy_k = \tilde y_i \degeneracy_k = \tilde z_i'.\]
For \(i \in \descendinginterval{l - 2}{0}\), we get
\begin{align*}
\tilde z_i & = \prod_{j \in \ascendinginterval{i + 1}{l - 2}} (\tilde z_j^{- 1} \face_{\descendinginterval{j}{i + 1}} \degeneracy_{\ascendinginterval{i}{j - 1}}) \prod_{j \in \descendinginterval{l - 2}{i}} (h_j \face_{\descendinginterval{l - 1}{i + 1}} \degeneracy_{\ascendinginterval{i}{l - 2}}) \\
& = \prod_{j \in \ascendinginterval{i + 1}{l - 2}} ({\tilde z_j'}{}^{- 1} \face_{\descendinginterval{j}{i + 1}} \degeneracy_{\ascendinginterval{i}{j - 1}}) \prod_{j \in \descendinginterval{l - 2}{i}} (h_j \face_{\descendinginterval{l - 1}{i + 1}} \degeneracy_{\ascendinginterval{i}{l - 2}}) \\
& = \prod_{j \in \ascendinginterval{i + 1}{l - 2}} (\tilde y_j^{- 1} \degeneracy_k \face_{\descendinginterval{j}{i + 1}} \degeneracy_{\ascendinginterval{i}{j - 1}}) \prod_{j \in \descendinginterval{l - 2}{i}} (g_{n, j} \degeneracy_k \face_{\descendinginterval{l - 1}{i + 1}} \degeneracy_{\ascendinginterval{i}{l - 2}}) \\
& = \Big( \prod_{j \in \ascendinginterval{i + 1}{l - 2}} (\tilde y_j^{- 1} \face_{\descendinginterval{j}{i + 1}} \degeneracy_{\ascendinginterval{i}{j - 1}}) \prod_{j \in \descendinginterval{l - 2}{i}} (g_{n, j} \face_{\descendinginterval{l - 1}{i + 1}} \degeneracy_{\ascendinginterval{i}{l - 2}}) \Big) \degeneracy_k = \tilde y_i \degeneracy_k.
\end{align*}

Now we suppose that \(k = l - 1\). For \(i \in \descendinginterval{n}{k + 1}\), we calculate
\[\tilde z_i = h_i = g_{n, i - 1} \degeneracy_k = \tilde y_{i - 1} \degeneracy_k = \tilde z_i'.\]
For \(i = k\), we get
\[\tilde z_k = h_k \face_{k + 1} \degeneracy_k = 1 = \tilde z_k'.\]
For \(i \in \descendinginterval{k - 1}{0}\), we get
\begin{align*}
\tilde z_i & = \prod_{j \in \ascendinginterval{i + 1}{k}} (\tilde z_j^{- 1} \face_{\descendinginterval{j}{i + 1}} \degeneracy_{\ascendinginterval{i}{j - 1}}) \prod_{j \in \descendinginterval{k}{i}} (h_j \face_{\descendinginterval{k + 1}{i + 1}} \degeneracy_{\ascendinginterval{i}{k}}) \\
& = \prod_{j \in \ascendinginterval{i + 1}{k}} ({\tilde z_j'}{}^{- 1} \face_{\descendinginterval{j}{i + 1}} \degeneracy_{\ascendinginterval{i}{j - 1}}) \prod_{j \in \descendinginterval{k}{i}} (h_j \face_{\descendinginterval{k + 1}{i + 1}} \degeneracy_{\ascendinginterval{i}{k}}) \\
& = \prod_{j \in \ascendinginterval{i + 1}{k - 1}} (\tilde y_j^{- 1} \degeneracy_k \face_{\descendinginterval{j}{i + 1}} \degeneracy_{\ascendinginterval{i}{j - 1}}) \prod_{j \in \descendinginterval{k - 1}{i}} (g_{n, j} \degeneracy_k \face_{\descendinginterval{k + 1}{i + 1}} \degeneracy_{\ascendinginterval{i}{k}}) \\
& = \Big( \prod_{j \in \ascendinginterval{i + 1}{k - 1}} (\tilde y_j^{- 1} \face_{\descendinginterval{j}{i + 1}} \degeneracy_{\ascendinginterval{i}{j - 1}}) \prod_{j \in \descendinginterval{k - 1}{i}} (g_{n, j} \face_{\descendinginterval{k}{i + 1}} \degeneracy_{\ascendinginterval{i}{k - 1}}) \Big) \degeneracy_k = \tilde y_i \degeneracy_k.
\end{align*}

At last, we suppose that \(k \in \descendinginterval{l - 2}{0}\). For \(i \in \descendinginterval{n}{l}\), we have
\[\tilde z_i = h_i = g_{n, i - 1} \degeneracy_k = \tilde y_{i - 1} \degeneracy_k = \tilde z_i'.\]
For \(i \in \descendinginterval{l - 1}{k + 1}\), we get
\begin{align*}
\tilde z_i & = \prod_{j \in \ascendinginterval{i + 1}{l - 1}} (\tilde z_j^{- 1} \face_{\descendinginterval{j}{i + 1}} \degeneracy_{\ascendinginterval{i}{j - 1}}) \prod_{j \in \descendinginterval{l - 1}{i}} (h_j \face_{\descendinginterval{l}{i + 1}} \degeneracy_{\ascendinginterval{i}{l - 1}}) \\
& = \prod_{j \in \ascendinginterval{i + 1}{l - 1}} ({\tilde z_j'}{}^{- 1} \face_{\descendinginterval{j}{i + 1}} \degeneracy_{\ascendinginterval{i}{j - 1}}) \prod_{j \in \descendinginterval{l - 1}{i}} (h_j \face_{\descendinginterval{l}{i + 1}} \degeneracy_{\ascendinginterval{i}{l - 1}}) \\
& = \prod_{j \in \ascendinginterval{i + 1}{l - 1}} (\tilde y_{j - 1}^{- 1} \degeneracy_k \face_{\descendinginterval{j}{i + 1}} \degeneracy_{\ascendinginterval{i}{j - 1}}) \prod_{j \in \descendinginterval{l - 1}{i}} (g_{n, j - 1} \degeneracy_k \face_{\descendinginterval{l}{i + 1}} \degeneracy_{\ascendinginterval{i}{l - 1}}) \\
& = (\prod_{j \in \ascendinginterval{i}{l - 2}} (\tilde y_j^{- 1} \face_{\descendinginterval{j}{i}} \degeneracy_{\ascendinginterval{i - 1}{j - 1}}) \prod_{j \in \descendinginterval{l - 2}{i - 1}} (g_{n, j} \face_{\descendinginterval{l - 1}{i}} \degeneracy_{\ascendinginterval{i - 1}{l - 2}})) \degeneracy_k = \tilde y_{i - 1} \degeneracy_k = \tilde z_i'.
\end{align*}
For \(i = k\), we have
\begin{align*}
\tilde z_k & = \prod_{j \in \ascendinginterval{k + 1}{l - 1}} (\tilde z_j^{- 1} \face_{\descendinginterval{j}{k + 1}} \degeneracy_{\ascendinginterval{k}{j - 1}}) \prod_{j \in \descendinginterval{l - 1}{k}} (h_j \face_{\descendinginterval{l}{k + 1}} \degeneracy_{\ascendinginterval{k}{l - 1}}) \\
& = \prod_{j \in \ascendinginterval{k + 1}{l - 1}} ({\tilde z_j'}{}^{- 1} \face_{\descendinginterval{j}{k + 1}} \degeneracy_{\ascendinginterval{k}{j - 1}}) \prod_{j \in \descendinginterval{l - 1}{k}} (h_j \face_{\descendinginterval{l}{k + 1}} \degeneracy_{\ascendinginterval{k}{l - 1}}) \\
& = \prod_{j \in \ascendinginterval{k + 1}{l - 1}} (\tilde y_{j - 1}^{- 1} \degeneracy_k \face_{\descendinginterval{j}{k + 1}} \degeneracy_{\ascendinginterval{k}{j - 1}}) \prod_{j \in \descendinginterval{l - 1}{k + 1}} (g_{n, j - 1} \degeneracy_k \face_{\descendinginterval{l}{k + 1}} \degeneracy_{\ascendinginterval{k}{l - 1}}) \\
& = (\tilde y_k^{- 1} \degeneracy_k) (\prod_{j \in \ascendinginterval{k + 1}{l - 2}} (\tilde y_j^{- 1} \face_{\descendinginterval{j}{k + 1}} \degeneracy_{\ascendinginterval{k}{j}})) (\prod_{j \in \descendinginterval{l - 2}{k}} (g_{n, j} \face_{\descendinginterval{l - 1}{k + 1}} \degeneracy_{\ascendinginterval{k}{l - 1}})) \\
& = (\tilde y_k^{- 1} \degeneracy_k) \Big( \Big( \prod_{j \in \ascendinginterval{k + 1}{l - 2}} (\tilde y_j^{- 1} \face_{\descendinginterval{j}{k + 1}} \degeneracy_{\ascendinginterval{k}{j - 1}}) \prod_{j \in \descendinginterval{l - 2}{k}} (g_{n, j} \face_{\descendinginterval{l - 1}{k + 1}} \degeneracy_{\ascendinginterval{k}{l - 2}}) \Big) \degeneracy_k \Big) = (\tilde y_k^{- 1} \degeneracy_k) (\tilde y_k \degeneracy_k) = 1 \\
& = \tilde z_k'.
\end{align*}
For \(i \in \descendinginterval{k - 1}{0}\), we get
\begin{align*}
\tilde z_i & = \prod_{j \in \ascendinginterval{i + 1}{l - 1}} (\tilde z_j^{- 1} \face_{\descendinginterval{j}{i + 1}} \degeneracy_{\ascendinginterval{i}{j - 1}}) \prod_{j \in \descendinginterval{l - 1}{i}} (h_j \face_{\descendinginterval{l}{i + 1}} \degeneracy_{\ascendinginterval{i}{l - 1}}) \\
& = \prod_{j \in \ascendinginterval{i + 1}{l - 1}} ({\tilde z_j'}{}^{- 1} \face_{\descendinginterval{j}{i + 1}} \degeneracy_{\ascendinginterval{i}{j - 1}}) \prod_{j \in \descendinginterval{l - 1}{i}} (h_j \face_{\descendinginterval{l}{i + 1}} \degeneracy_{\ascendinginterval{i}{l - 1}}) \\
& = \prod_{j \in \ascendinginterval{i + 1}{k - 1}} (\tilde y_j^{- 1} \degeneracy_k \face_{\descendinginterval{j}{i + 1}} \degeneracy_{\ascendinginterval{i}{j - 1}}) \prod_{j \in \ascendinginterval{k + 1}{l - 1}} (\tilde y_{j - 1}^{- 1} \degeneracy_k \face_{\descendinginterval{j}{i + 1}} \degeneracy_{\ascendinginterval{i}{j - 1}}) \\
& \qquad \cdot \prod_{j \in \descendinginterval{l - 1}{k + 1}} (g_{n, j - 1} \degeneracy_k \face_{\descendinginterval{l}{i + 1}} \degeneracy_{\ascendinginterval{i}{l - 1}}) \prod_{j \in \descendinginterval{k - 1}{i}} (g_{n, j} \degeneracy_k \face_{\descendinginterval{l}{i + 1}} \degeneracy_{\ascendinginterval{i}{l - 1}}) \\
& = \prod_{j \in \ascendinginterval{i + 1}{k - 1}} (\tilde y_j^{- 1} \degeneracy_k \face_{\descendinginterval{j}{i + 1}} \degeneracy_{\ascendinginterval{i}{j - 1}}) \prod_{j \in \ascendinginterval{k}{l - 2}} (\tilde y_j^{- 1} \degeneracy_k \face_{\descendinginterval{j + 1}{i + 1}} \degeneracy_{\ascendinginterval{i}{j}}) \prod_{j \in \descendinginterval{l - 2}{i}} (g_{n, j} \degeneracy_k \face_{\descendinginterval{l}{i + 1}} \degeneracy_{\ascendinginterval{i}{l - 1}}) \\
& = \Big( \prod_{j \in \ascendinginterval{i + 1}{l - 2}} (\tilde y_j^{- 1} \face_{\descendinginterval{j}{i + 1}} \degeneracy_{\ascendinginterval{i}{j - 1}}) \prod_{j \in \descendinginterval{l - 2}{i}} (g_{n, j} \face_{\descendinginterval{l - 1}{i + 1}} \degeneracy_{\ascendinginterval{i}{l - 2}}) \Big) \degeneracy_k = \tilde y_i \degeneracy_k = \tilde z_i'.
\end{align*}

Altogether, we obtain a simplicial map
\[\DiagonalFunctor \Nerve G \times \StandardSimplex{1} \morphism[H] \DiagonalFunctor \Nerve G.\]

To prove that \(H\) is a simplicial homotopy from \(D_G S_G\) to \(\id_{\DiagonalFunctor \Nerve G}\), it remains to show that \(\ins_0 H = D_G S_G\) and \(\ins_1 H = \id_{\DiagonalFunctor \Nerve G}\). For \(n \in \Naturals_0\), \(k \in [0, n + 1]\), \((g_{n, i})_{i \in \descendinginterval{n - 1}{0}} \in \DiagonalFunctor_n \Nerve G\), \(n \in \Naturals_0\), we have
\[(g_{n, i})_{i \in \descendinginterval{n - 1}{0}} (D_G)_n (S_G)_n = (g_{n, i} \face_{\descendinginterval{n}{i + 1}})_{i \in \descendinginterval{n - 1}{0}} (S_G)_n = (y_i)_{i \in \descendinginterval{n - 1}{0}}\]
with
\begin{align*}
y_i & = \prod_{j \in \ascendinginterval{i + 1}{n - 1}} (y_j^{- 1} \face_{\descendinginterval{j}{i + 1}} \degeneracy_{\ascendinginterval{i}{j - 1}}) \prod_{j \in \descendinginterval{n - 1}{i}} (g_{n, j} \face_{\descendinginterval{n}{j + 1}} \face_{\descendinginterval{j}{i + 1}} \degeneracy_{\ascendinginterval{i}{n - 1}}) \\
& = \prod_{j \in \ascendinginterval{i + 1}{n - 1}} (y_j^{- 1} \face_{\descendinginterval{j}{i + 1}} \degeneracy_{\ascendinginterval{i}{j - 1}}) \prod_{j \in \descendinginterval{n - 1}{i}} (g_{n, j} \face_{\descendinginterval{n}{i + 1}} \degeneracy_{\ascendinginterval{i}{n - 1}})
\end{align*}
for \(i \in \descendinginterval{n - 1}{0}\), and
\[((g_{n, i})_{i \in \descendinginterval{n - 1}{0}}, \cotime^{n + 1 - k}) H_n = (y_i^{(n + 1 - k)})_{i \in \descendinginterval{n - 1}{0}}\]
with
\[y_i^{(n + 1 - k)} := \begin{cases} g_{n, i} & \text{for } i \in \descendinginterval{n - 1}{k - 1} \cap \Naturals_0, \\ \prod_{j \in \ascendinginterval{i + 1}{k - 2}} ((y_j^{(n + 1 - k)})^{- 1} \face_{\descendinginterval{j}{i + 1}} \degeneracy_{\ascendinginterval{i}{j - 1}}) \\ \qquad \cdot \prod_{j \in \descendinginterval{k - 2}{i}} (g_{n, j} \face_{\descendinginterval{k - 1}{i + 1}} \degeneracy_{\ascendinginterval{i}{k - 2}}) & \text{for } i \in \descendinginterval{k - 2}{0}. \end{cases}\]
But by descending induction on \(i \in \descendinginterval{n - 1}{0}\), we get
\begin{align*}
y_i & = \prod_{j \in \ascendinginterval{i + 1}{n - 1}} (y_j^{- 1} \face_{\descendinginterval{j}{i + 1}} \degeneracy_{\ascendinginterval{i}{j - 1}}) \prod_{j \in \descendinginterval{n - 1}{i}} (g_{n, j} \face_{\descendinginterval{n}{i + 1}} \degeneracy_{\ascendinginterval{i}{n - 1}}) \\
& = \prod_{j \in \ascendinginterval{i + 1}{n - 1}} ((y_j^{(0)})^{- 1} \face_{\descendinginterval{j}{i + 1}} \degeneracy_{\ascendinginterval{i}{j - 1}}) \prod_{j \in \descendinginterval{n - 1}{i}} (g_{n, j} \face_{\descendinginterval{n}{i + 1}} \degeneracy_{\ascendinginterval{i}{n - 1}}) = y_i^{(0)}.
\end{align*}
Hence the simplicial map \(H\) fulfills
\begin{align*}
(g_{n, i})_{i \in \descendinginterval{n - 1}{0}} (\ins_0)_n H_n & = ((g_{n, i})_{i \in \descendinginterval{n - 1}{0}}, \cotime^0) H_n = (y_i^{(0)})_{i \in \descendinginterval{n - 1}{0}} \\
& = (y_i)_{i \in \descendinginterval{n - 1}{0}} = (g_{n, i})_{i \in \descendinginterval{n - 1}{0}} (D_G)_n (S_G)_n
\end{align*}
and
\[(g_{n, i})_{i \in \descendinginterval{n - 1}{0}} (\ins_1)_n H_n = ((g_{n, i})_{i \in \descendinginterval{n - 1}{0}}, \cotime^{n + 1}) H_n = (g_{n, i})_{i \in \descendinginterval{n - 1}{0}}\]
for each \((g_{n, i})_{i \in \descendinginterval{n - 1}{0}} \in \DiagonalFunctor_n \Nerve G\), \(n \in \Naturals_0\).

In order to prove that \(\KanClassifyingSimplicialSet G\) is a strong deformation retract of \(\DiagonalFunctor \Nerve G\), it remains to show that \(H\) is constant along \(S_G\). Concretely, this means the following. For \((g_i)_{i \in \descendinginterval{n - 1}{0}} \in \KanClassifyingSimplicialSet_n G\), we have
\[((g_i)_{i \in \descendinginterval{n - 1}{0}} (S_G)_n, \cotime^{n + 1 - k}) H_n = ((y_i)_{i \in \descendinginterval{n - 1}{0}}, \cotime^{n + 1 - k}) H_n = (y_i^{(n + 1 - k)})_{i \in \descendinginterval{n - 1}{0}},\]
where
\[y_i := \prod_{j \in \ascendinginterval{i + 1}{n - 1}} (y_j^{- 1} \face_{\descendinginterval{j}{i + 1}} \degeneracy_{\ascendinginterval{i}{j - 1}}) \prod_{j \in \descendinginterval{n - 1}{i}} (g_j \face_{\descendinginterval{j}{i + 1}} \degeneracy_{\ascendinginterval{i}{n - 1}})\]
and
\[y_i^{(n + 1 - k)} := \begin{cases} y_i & \text{for } i \in \descendinginterval{n - 1}{k - 1} \cap \Naturals_0, \\ \prod_{j \in \ascendinginterval{i + 1}{k - 2}} ((y_j^{(n + 1 - k)})^{- 1} \face_{\descendinginterval{j}{i + 1}} \degeneracy_{\ascendinginterval{i}{j - 1}}) & \\ \qquad \cdot \prod_{j \in \descendinginterval{k - 2}{i}} (y_j \face_{\descendinginterval{k - 1}{i + 1}} \degeneracy_{\ascendinginterval{i}{k - 2}}) & \text{for } i \in \descendinginterval{k - 2}{0}. \end{cases}\]
Now, we have to show that \(y_i^{(n + 1 - k)} = y_i\) for all \(i \in \descendinginterval{n - 1}{0}\), \(k \in [0, n + 1]\). For \(k \in \{n + 1, 0\}\), this follows since \(H\) is a simplicial homotopy from \(D_G S_G\) to \(\id_{\DiagonalFunctor \Nerve G}\) and since \(S_G D_G S_G = S_G\). So we may assume that \(k \in \descendinginterval{n}{1}\) and have to show that \(y_i^{(n + 1 - k)} = y_i\) for every \(i \in \descendinginterval{k - 2}{0}\). But we have
\begin{align*}
& y_i \face_{\descendinginterval{k - 1}{i + 1}} \degeneracy_{\ascendinginterval{i}{k - 2}} = \Big( \prod_{j \in \ascendinginterval{i + 1}{n - 1}} (y_j^{- 1} \face_{\descendinginterval{j}{i + 1}} \degeneracy_{\ascendinginterval{i}{j - 1}}) \prod_{j \in \descendinginterval{n - 1}{i}} (g_j \face_{\descendinginterval{j}{i + 1}} \degeneracy_{\ascendinginterval{i}{n - 1}}) \Big) \face_{\descendinginterval{k - 1}{i + 1}} \degeneracy_{\ascendinginterval{i}{k - 2}} \\
& = \prod_{j \in \ascendinginterval{i + 1}{n - 1}} (y_j^{- 1} \face_{\descendinginterval{j}{i + 1}} \degeneracy_{\ascendinginterval{i}{j - 1}} \face_{\descendinginterval{k - 1}{i + 1}} \degeneracy_{\ascendinginterval{i}{k - 2}}) \prod_{j \in \descendinginterval{n - 1}{i}} (g_j \face_{\descendinginterval{j}{i + 1}} \degeneracy_{\ascendinginterval{i}{n - 1}} \face_{\descendinginterval{k - 1}{i + 1}} \degeneracy_{\ascendinginterval{i}{k - 2}}) \\
& = \prod_{j \in \ascendinginterval{i + 1}{k - 1}} (y_j^{- 1} \face_{\descendinginterval{j}{i + 1}} \degeneracy_{\ascendinginterval{i}{j - 1}} \face_{\descendinginterval{k - 1}{j + 1}} \face_{\descendinginterval{j}{i + 1}} \degeneracy_{\ascendinginterval{i}{k - 2}}) \\
& \qquad \cdot \prod_{j \in \ascendinginterval{k}{n - 1}} (y_j^{- 1} \face_{\descendinginterval{j}{i + 1}} \degeneracy_{\ascendinginterval{i}{k - 1}} \degeneracy_{\ascendinginterval{k}{j - 1}} \face_{\descendinginterval{k - 1}{i + 1}} \degeneracy_{\ascendinginterval{i}{k - 2}}) \\
& \qquad \cdot \prod_{j \in \descendinginterval{n - 1}{i}} (g_j \face_{\descendinginterval{j}{i + 1}} \degeneracy_{\ascendinginterval{i}{k - 1}} \degeneracy_{\ascendinginterval{k}{n - 1}} \face_{\descendinginterval{k - 1}{i + 1}} \degeneracy_{\ascendinginterval{i}{k - 2}}) \\
& = \prod_{j \in \ascendinginterval{i + 1}{k - 1}} (y_j^{- 1} \face_{\descendinginterval{j}{i + 1}} \face_{\descendinginterval{i + k - 1 - j}{i + 1}} \degeneracy_{\ascendinginterval{i}{j - 1}} \face_{\descendinginterval{j}{i + 1}} \degeneracy_{\ascendinginterval{i}{k - 2}}) \\
& \qquad \cdot \prod_{j \in \ascendinginterval{k}{n - 1}} (y_j^{- 1} \face_{\descendinginterval{j}{i + 1}} \degeneracy_{\ascendinginterval{i}{k - 1}} \face_{\descendinginterval{k - 1}{i + 1}} \degeneracy_{\ascendinginterval{i + 1}{i + j - k}} \degeneracy_{\ascendinginterval{i}{k - 2}}) \\
& \qquad \cdot \prod_{j \in \descendinginterval{n - 1}{i}} (g_j \face_{\descendinginterval{j}{i + 1}} \degeneracy_{\ascendinginterval{i}{k - 1}} \face_{\descendinginterval{k - 1}{i + 1}} \degeneracy_{\ascendinginterval{i + 1}{i + n - k}} \degeneracy_{\ascendinginterval{i}{k - 2}}) \\
& = \prod_{j \in \ascendinginterval{i + 1}{k - 1}} (y_j^{- 1} \face_{\descendinginterval{k - 1}{i + 1}} \degeneracy_{\ascendinginterval{i}{k - 2}}) \prod_{j \in \ascendinginterval{k}{n - 1}} (y_j^{- 1} \face_{\descendinginterval{j}{i + 1}} \degeneracy_{\ascendinginterval{i}{j - 1}}) \prod_{j \in \descendinginterval{n - 1}{i}} (g_j \face_{\descendinginterval{j}{i + 1}} \degeneracy_{\ascendinginterval{i}{n - 1}}),
\end{align*}
and this implies, by induction on \(i \in \descendinginterval{k - 2}{0}\), that
\begin{align*}
y_i^{(n + 1 - k)} & = \prod_{j \in \ascendinginterval{i + 1}{k - 2}} ((y_j^{(n + 1 - k)})^{- 1} \face_{\descendinginterval{j}{i + 1}} \degeneracy_{\ascendinginterval{i}{j - 1}}) \prod_{j \in \descendinginterval{k - 2}{i}} (y_j \face_{\descendinginterval{k - 1}{i + 1}} \degeneracy_{\ascendinginterval{i}{k - 2}}) \\
& = \prod_{j \in \ascendinginterval{i + 1}{k - 2}} (y_j^{- 1} \face_{\descendinginterval{j}{i + 1}} \degeneracy_{\ascendinginterval{i}{j - 1}}) \prod_{j \in \descendinginterval{k - 2}{i}} (y_j \face_{\descendinginterval{k - 1}{i + 1}} \degeneracy_{\ascendinginterval{i}{k - 2}}) \\
& = \prod_{j \in \ascendinginterval{i + 1}{k - 2}} (y_j^{- 1} \face_{\descendinginterval{j}{i + 1}} \degeneracy_{\ascendinginterval{i}{j - 1}}) \prod_{j \in \descendinginterval{k - 2}{i + 1}} (y_j \face_{\descendinginterval{k - 1}{i + 1}} \degeneracy_{\ascendinginterval{i}{k - 2}}) \\
& \qquad \cdot \prod_{j \in \ascendinginterval{i + 1}{k - 1}} (y_j^{- 1} \face_{\descendinginterval{k - 1}{i + 1}} \degeneracy_{\ascendinginterval{i}{k - 2}}) \prod_{j \in \ascendinginterval{k}{n - 1}} (y_j^{- 1} \face_{\descendinginterval{j}{i + 1}} \degeneracy_{\ascendinginterval{i}{j - 1}}) \prod_{j \in \descendinginterval{n - 1}{i}} (g_j \face_{\descendinginterval{j}{i + 1}} \degeneracy_{\ascendinginterval{i}{n - 1}}) \\
& = \prod_{j \in \ascendinginterval{i + 1}{n - 1}} (y_j^{- 1} \face_{\descendinginterval{j}{i + 1}} \degeneracy_{\ascendinginterval{i}{j - 1}}) \prod_{j \in \descendinginterval{n - 1}{i}} (g_j \face_{\descendinginterval{j}{i + 1}} \degeneracy_{\ascendinginterval{i}{n - 1}}) = y_i
\end{align*}
for all \(i \in \descendinginterval{k - 2}{0}\).
\end{proof}

\bigskip

{\raggedleft Sebastian Thomas \\ Lehrstuhl D für Mathematik \\ RWTH Aachen \\ Templergraben 64 \\ D-52062 Aachen \\ sebastian.thomas@math.rwth-aachen.de \\ \url{http://www.math.rwth-aachen.de/~Sebastian.Thomas/} \\}

\end{document}